\documentclass[preprint,12pt]{article}
\usepackage{epsfig}
\usepackage{amsmath}
\usepackage{amssymb}
\usepackage{amsthm}
\usepackage{multicol}
\usepackage{authblk}
\theoremstyle{plain}
\newtheorem{theorem}{Theorem}[section]

\newtheorem{lemma}[theorem]{Lemma}

\theoremstyle{remark}

\theoremstyle{definition}
\newtheorem{definition}[theorem]{Definition}
\newcommand{\R}{\mathbb{R}}

\title{Application of the semigroup  theory to a combustion problem in a multi-layer porous medium}

\author[1]{E. A. Alarcon\thanks{Instituto de Matem\'atica e Estat\'istica-IME - Universidade Federal de Goi\'as-UFG, Campus II, Goi\^ania , GO, Brazil (In Memoriam)}, M. R. Batista\thanks{Instituto Federal de Goi\'as-IFG, Campus Goi\^ania, GO, Brazil (e-mail: marcos.batista@ifg.edu.br )}, A. Cunha\thanks{Instituto de Matem\'atica e Estat\'istica - Universidade Federal de Goi\'as, Campus II, Goi\^ania , GO, Brazil (e-mail: alysson@ufg.br)}\and J. C. Da Mota\thanks{Instituto de Matem\'atica e Estat\'istica - Universidade Federal de Goi\'as, Campus II, Goi\^ania , GO, Brazil (e-mail: jesusdamota@gmail.com )}, R. A. Santos\thanks{Instituto de Matem\'atica e Estat\'istica - Universidade Federal de Goi\'as, Campus II, Goi\^ania , GO, Brazil (e-mail: rasantos@ufg.br)}}

\begin{document}
\maketitle

\begin{abstract} 
	
This study proved that the Cauchy problem for 
a one-dimensional reaction-diffusion-convection system is locally and globally well-posed in $\mathtt{H}^2(\R)$. 
The system modeled a gasless combustion front through a multi-layer porous medium when the fuel concentration in each layer was a known function. Combustion has critical practical porous media applications, such as in in-situ combustion processes in oil reservoirs and several other areas.

Earlier studies considered physical parameters (e.g., porosity, thermal conductivity, heat capacity, and initial fuel concentration ) constant. Here, we consider a more realistic model where these parameters are functions of the spatial variable rather than constants. Furthermore, in previous studies, we did not consider the continuity of the solution regarding the initial data and parameters, unlike the current study. This proof uses a novel approach to combustion problems in porous media. We follow the abstract semigroups theory of operators in the Hilbert space and the well-known Kato's theory for a well-posed associated initial value problem.
	
\end{abstract}

{\bf Keywords:}	Reaction-diffusion-convection system,	Combustion, Multilayer porous medium,  
 Semigroups theory,  Kato's theory, Well-posedness solution.
	
{\bf MSC:} 35K51, 35K57, 76S05, 80A25.

\section{Introduction}\label{introduction}
We can find a short review of combustion applications in porous media in 
\cite{Banerjee, Gharehghani, Mujeebu}. Among the several applications is the in-situ combustion process for oil recovery from a petroleum reservoir \cite{Ado, Ameli, Gutierrez, Sarathi}. Many petroleum reservoirs with crude oil have layers with different physical properties \cite{Gao, Lefkovits}. Motivated by these applications,   
we consider a one-dimensional model for the combustion of oxygen and solid fuel in a porous medium with $n$ $(n \ge 2)$ parallel layers. Full derivation of the model
is given by \cite{Batista1}: With some simplifying assumptions, such as incompressibility, the model reduces to a simple system of $n$ reaction-diffusion-convection equations, coupled with $n$ ordinary differential equations, governing the temperature $u_i = u_i(x, t)$ and the fuel concentration $y_i = y_i(x, t)$ in layer $i$ $(i = 1,\ldots, n)$, where $x \in \R$ is the space variable and $t \ge 0$ is the time. 

The model is a generalization of the two-layer model derived in \cite{DaMota-Schecter}, where traveling wave solutions were found. \cite{DaMota-Marcelo-Ronaldo} solved the Cauchy problem for this two-layer model.
Traveling waves have been observed in several practical problems 
\cite{Chen, Song, Zhen}, including some other combustion problems \cite{Chapiro, DaMota-Cido, Ghazaryan}.

In \cite{Batista1}, the existence and uniqueness of the classical local and global solutions over time 
were proven for an initial and boundary value problem for the simplified model, where the fuel concentrations in the layers are known functions. The main tool used is the monotonous iterative method for the upper and lower solutions. Recently, in \cite{Batista2}, by using Schauder's fixed-point theorem,
a classical solution was proven to exist for the complete model's initial and boundary value problem, where the fuel concentrations are also unknown.
The study considered physical properties (e.g., porosity, conductivity, heat capacity, and initial fuel concentration) constant in each layer. 

We also assumed that the fuel concentration in each layer is a known function; however, we considered a more realistic model, where the physical properties mentioned above depend on the spatial variable $x$. We proved the existence and uniqueness of local and global classical solutions in time for an initial value problem (Cauchy problem) associated with the model. We also proved continuous dependence for the initial data and parameters in the sense of Theorem~\ref{combusiton-continuous-dependence-theo} below. Thus, the following initial value problem is well-posed in $\mathtt{H}^2(\R)$.
\begin{align}\label{combustion-problem}
\left\{
\begin{array}{l}
u_t + L(t)u = f(x, t, u), \,\,\, x \in \R, \,\,\, t >0,  \\
u(x, 0) = \phi(x). 
\end{array}
\right.
\end{align}
Here, $u = (u_1, \ldots, u_n)$ is an unknown vector of temperatures 
$\phi = (\phi_1, \ldots, \phi_n)$ is the given vector of the initial temperatures, and $L(t)$ is the
partial differential operator defined by
\begin{align}\label{operatorL} 
L(t)u = \big(L_1(t)u_1, \ldots, L_n(t)u_n \big),
\end{align}
where
\begin{equation}\label{operatorLi}
L_i(t)u_i := -\alpha_i(x, t)\,\partial_{xx}u_i + \beta_i(x, t)\,\partial_{x}u_i,\,\,\,
i = 1, \ldots, n,
\end{equation}
with coefficients
\begin{equation}\label{alpha-beta}
\alpha_i(x, t) = \dfrac{\lambda_i(x)}{a_i(x) + b_i(x)\,y_i(x,\,t)}, \,\,\, 
\beta_i(x, t) = \dfrac{c_i(x)}{a_i(x) + b_i(x)\,y_i(x,\,t)}\,.
\end{equation}
Functions $a_i$, $b_i$, $c_i$, and $\lambda_i$ for $i = 1, \ldots, n$ are defined depending on the physical parameters (see \cite{Batista1}), which  are known functions of the spatial variable $x$. The fuel concentration $y_i$ is known to be nonnegative and bounded\footnote{We can reasonably assume that fuel concentrations $y_i$ take their values in the interval $[0, 1]$, as this is physically expected, in view of the concentration definition.} function of $(x, t)$ as follows:  
The combustion reaction rate, heat transfer between two adjacent layers, and heat loss to the external medium are all included in the source function $f = (f_1, \ldots, f_n)$, the components of which are defined by
\begin{align} \label{raction-function} 
&f_1(x, t, u) = \frac{-(c_1)_x\,u_1 }{a_1 + b_1 y_1} + 
\frac{(K_1 b_1 u_1+d_1)y_1\,g(u_1)}{a_1 + b_1 y_1} + \frac{q_1(u_2 - u_1)}{a_1 + b_1 y_1}
\nonumber
\\   
& \hspace{0.8in} - \frac{\overline q_1(u_1 - u_e)}{a_1+b_1 y_1}\,,
\nonumber \\
&f_i(x, t, u) = \frac{-(c_i)_x\,u_i }{a_i + b_i y_i} + 
\frac{(K_i b_i u_i+d_i)y_i\,g(u_i)}{a_i + b_i y_i} - \frac{q_{i-1}(u_i - u_{i-1})}{a_i + b_i y_i}
\\
& \hspace{0.8in} + \frac{q_i(u_{i+1} - u_i)}{a_i + b_i y_i}, \quad i = 2, \ldots, n-1\,,
\nonumber \\
&f_n(x, t, u) = \frac{-(c_n)_x\,u_n }{a_n + b_n y_n} + 
\frac{(K_n b_n u_n+d_n)y_n\,g(u_n)}{a_n + b_n y_n} - \frac{q_{n-1}(u_n - u_{n-1})}{a_n + b_n y_n}
\nonumber
\\
& \hspace{0.8in} - \frac{\overline q_2(u_n - u_e)}{a_n + b_n y_n}\,, 
\nonumber 
\end{align}
where $d_i$ is also a known function of the spatial variable $x$, and $g$ is a function related to the Arrhenius law given by
\begin{equation}\label{darcy} 
g(\theta)=\left\{\begin{array}{c}
e^{-\frac{E}{\theta}},\;\;\mbox{se}\;\; \theta>0\\
0,\;\;\mbox{se}\;\; \theta\leq 0\,.
\end{array}\right.
\end{equation}
The other quantities $K_i$, $q_i$, $\overline q_1$, $\overline q_2$, $E$ are non-negative 
parameters, as defined in \cite{Batista1}, and $u_e$ denotes the temperature of the external environment, which is constant.

Throughout this study, \eqref{combustion-problem} is referred to as a combustion problem. Index~$i$ refers to layer $i$; unless stated otherwise, $i = 1, \ldots, n$.
To facilitate notation, when there is no doubt, we omit the variable $x$ from $u(x, t)$, $\phi(x)$, $y(x,t)$, and $f(x, t, u)$ by writing simply $u(t)$, $\phi$, $y(t)$, and $f(t, u)$.

The solutions treated here are conventional according to the following definition: 
\begin{definition}\label{def-solution}	
A local solution to problem \eqref{combustion-problem} is a function
$u = (u_1,\ldots, u_2)$ $\in C\big([0,\,T], \,W \big) \cap C^1\big( [0,\,T], \,\mathtt{L}^2(\R)^n \big)$,
which satisfies \eqref{combustion-problem} for some $T > 0$ and some open set 
$W \subset \mathtt{H}^2(\R)^n$. The solution is global in time if it satisfies \eqref{combustion-problem} for all $T > 0$: 
\end{definition}

Our main results are summarized in the next three theorems; these proofs rely on a new approach for combustion in porous media, which uses the semigroups theory of operators in the Hilbert space and the well-known Kato's theory for the initial value problem.

\begin{theorem} [Local solution]\label{combustion-local-theo} 
Assume that hypotheses (Hy6) and (Hy7) given in Section~\ref{combustion-local} are satisfied.
If $\phi = (\phi_1, \ldots, \phi_n) \in \mathtt{H}^2(\R)^n$, then the initial value problem \eqref{combustion-problem} has a unique local solution.
Each component of this solution is given in the following integral form:  
\begin{align}\label{integral-form1}
u_i(t) = U_i(t,\,0) \phi_i + \int_0^t\,U_i(t,\,\tau)\,f_i\big(\tau,\,u(\tau)\big)\,d\tau,
\end{align}
$t \in [0, T]$ for some $T > 0$, where $U_i$ is the evolution propagation operator associated with $L_i(t)$, defined by Lemma~\ref{propagador}.
\end{theorem}

\begin{theorem}[Global solution] \label{combustion-global-theo} 
We assume that the hypothesis (Hy8) given in Section~\ref{combustion-global} is satisfied. 
If $\phi = (\phi_1, \ldots, \phi_n) \in \mathtt{H}^2(\R)^n$, then the initial value problem \eqref{combustion-problem} has a unique global solution: Each component of this solution is given in the integral form, as in \eqref{integral-form1}, for any $T > 0$.
\end{theorem}

\begin{theorem}[Continuous dependence]\label{combusiton-continuous-dependence-theo}
Let us assume the same hypotheses as in Theorem~\ref{combustion-local-theo}. Then, the function that maps the initial data and the parameters into the solution given by this theorem is continuous 
in the $\mathtt{H}^2(\R)^n\text{-norm}$. Similarly, let us assume the same hypotheses as in Theorem~\ref{combustion-global-theo}; then, the analogous result holds.
\end{theorem}

Theorems \ref{combustion-local-theo}, \ref{combustion-global-theo}, and \ref{combusiton-continuous-dependence-theo} ensure that problem \eqref{combustion-problem} is locally and globally well posed.

The rest of this paper is organized as follows.

In Section~\ref{preliminaries}, we describe the preliminary results of this study. First, a list of notations and basic definitions is presented. Then, we present some results on the semigroup theory for an abstract linear evolution equation in Banach space, where we focus only on results that will be used to find the local solution to an initial value problem for a general second-order semi-linear scalar equation. This local solution is described in \ref{scalar-local}.
In \ref{scalar-global}, the maximal solution concept is extended to a global solution. 
A complete study of this subject can be found in a series of papers by Kato~ \cite{Kato1, Kato2, Kato3, Kato4, Kato5, Kato6}.

In Section~\ref{combustion-local}, following the proof of the general scalar equation, we prove Theorem~\ref{combustion-local-theo}. The main tool is based on the evolution operator for the closure 
of the $L_i(t)$ operator in $\mathtt{L}^2(\R)$, which is determined by Lemma~\ref{propagador-combustao}
for each $t \in [0, T]$.

In Section~\ref{combustion-global}, we prove Theorem~\ref{combustion-global-theo}, where 
the closure of $L_i(t)$ in $\mathtt{L}^2(\R)$ must be valid for each $t >0$.

In Section~\ref{combustion-continuous-dependence}, we prove Theorem~\ref{combusiton-continuous-dependence-theo}. First, we prove continuous dependence for only the initial data. Then, including the hypothesis that the parameters are functions of $\mathtt{H}^2(\R)$, we prove the continuous dependence for the initial data and parameters. In each case, the proof holds for any fixed $T>0$ ; it encompasses local and global continuity. 

\section{Preliminaries}\label{preliminaries}

This section presents the main notations used in this study and describes the abstract linear evolution problem results.

\subsection{Notations and general definitions}\label{notation}
The set of real numbers is denoted by $\R$, $I \subset \R$ is an interval, and $T$ is a 
positive number.
We denote $X$ and $Y$ Banach spaces such that $Y \subset X$.
The spaces of type $\mathtt{L}^p$ used in this study are $\mathtt{L}^1(\R)$, $\mathtt{L}^2(\R)$, and $\mathtt{L}^{\infty}(\R)$. The Sobolev spaces are $\mathtt{H}^1(\R)$ and $\mathtt{H}^2(\R)$.
\begin{itemize}
	\item[] $| \cdot |$ is the absolute value or Euclidean norm in $\R^n$. 
	 \item[] $\|\cdot\|_{X}$ is the norm on space  $X$. 
	\item[] $X^n = X \times \ldots \times X$, \,\,$n$ times, with the norm
     $\|u=(u_1, \ldots, u_n) \|_{X^n}	= \max\big\{\|u_i \|_X, \,\, i=1,\ldots, n  \big\}.$   	
	\item[] $<\,,\,>$ is the inner product in the Hilbert space.
	\item[] $\partial_{x}=\frac{\partial}{\partial x},\,\partial_{t}=\frac{\partial}{\partial t}$.
	\item[] $\Omega=\{(x,t);\, x \in \R,\,\,t \ge 0\}$.
	\item[] $\Omega_T=\{(x,t);\, x \in \R,\,\,0 \leq t \leq T, \,\,\, T > 0\}$.
	\item[] $C(I,\, X)$ is the space of continuous functions defined on $I$ into  $X$. 
	If $I$ is compact, then it is a Banach space with the supremum norm.
	\item[] $d(u,v)=\sup_{[0,T]}\|u(t)-v(t)\|_X$ is the metric in $C(I,\, X)$.
	\item[]   $C^1(I,\, X)$ is the space of continuously differentiable functions defined on $I$ into  $X$. 
	\item[] $C^{\infty}_0(\R)$ is the space of infinitely differentiable functions defined in $\R$ of compact support.
	\item[] $B(Y,X)$ \big($B(X) = B(X,X)$\big) is the space of all bounded linear operators from $Y$ to $X$ 
	with norm $\|\cdot\|_{B(Y,X)}$.
	\item[] $D(A)$ is the domain of the $A$ operator.
	\item[] $R(A)$ is the range of the $A$ operator.
    \item[] $\rho(A)$ is the resolvent of the $A$ operator.
    \item[] $\mathcal{S}(\R)$ is the Schwarz space of rapidly decreasing $C^{\infty}$ functions.
    \item[] For $g : [0,\,T] \rightarrow X$, and $A : [0,\,T] \rightarrow B(Y, X)$,
    \begin{align*}
    & \| g \|_{\infty, \,X} = \sup_{0 \leq t \leq T} \| g(t) \|_X , \,\,\, \| g \|_{1, \,X} = \int_0^T \| g(t) \|_X\,dt,\\
    & \| A \|_{\infty,\, B(Y, X)} = \sup_{0 \leq t \leq T} \| A(t) \|_{B(Y, X)}\,.
    \end{align*}
    \item[] $C\left(I,\, B(Y, X)\right)$ is the space all norm-continuous functions.
    \item[] $C_* \left(I,\, B(Y, X)\right)$ is the space of all strongly continuous functions on $Y$ to $X$.
    \item[] $F \in \mathtt{L}_*^\infty \left(I,\, B(Y, X)\right)$ implies that $F(t) \in B\left(Y, X\right)$ is defined for a.e. 
    $t \in I$, is strongly measurable (i.e., $F(t) y \in X$ is strongly measurable in $t$ for each $y \in Y$), and $\| F(t)\|_{B(Y, X)}$ is essentially bounded in $t$.
    \item[] $\partial_t F \in  \mathtt{L}_*^\infty \left(I,\, B(Y, X)\right)$, or equivalently
    $F \in \mathrm{Lip}_*\left(I,\, B(Y, X)\right)$, indicating that there is a function $G \in  \mathtt{L}_*^\infty \left(I,\, B(Y, X)\right)$ such that $F(t)$ is an indefinite strong integral of $G$, (i.e., $F(t) y = F(0) y + \int_0^t\, G(s) y\, ds$). Here, it is understood that $F(0)$ and hence $F(t)$ belong to 
    $B(Y, X)$. In this case, we write $\partial_t F = G$; thus,
    $G$ in uniquely determined by $F$ up to equivalence.
    \item[] $\mathbb{G}(X, M, \beta)$ is the set $\big\{ A(t): \,\,\, t \in I  \big\}$ of all stable families (see Definition \ref{stable}) of negative generators of semigroups of class $C_0$
     ($C_0$-semigroups) on $X$, with stability constants $M$ and $\beta$.
    
\end{itemize}

\subsection{Evolution operator}\label{abstract}
This section describes the evolution operator for the following abstract linear evolution problem:
\begin{equation}\label{evo11}
\left\{
\begin{array}{l}
\dfrac{d u}{dt} + A(t) u = f(t),\; t \in I=[0,\,T],   \\ \\
u(0)=\phi ,
\end{array}
\right.
\end{equation}
where $A(t) : D(A(t)) \subseteq X \to X$ and $t \in [0, T]$ is
a family of (generally unbounded) linear operators acting in Banach space $X$. 
The unknown $u=u(t)$ and inhomogeneous term $f(t)$ are functions of $I = [0,\,T]$ to $X$,
and $\phi \in Y \subseteq X$ is a given function.

The case in which $A(t)$ does not depend on $t$ belongs to the theory of one-parameter semigroups of operators (Hille-Yosida theory \cite{Pazy}).

\begin{definition}\label{stable}
	A family $\left\{ A(t) \colon t \in I = [0, T] \right\}$ of infinitesimal
	generators of $C_0$ semigroups on a Banach space $X$ is called {\it stable} if there are constants $M\geq 1$ and $\beta$ (called stability constants) such that
	$$
	(\beta,\,+\infty ) \subset \rho \left(A(t)\right), \;\;\; t\in [0,\,T].
	$$
	and
	$$
	\Big\|\prod^{k}_{j=1} R \big(\lambda; A(t_j)\big) \Big\| \leq M (\lambda - \beta)^{-k},\;\;\;\lambda > \beta,
	$$
where $ \rho \left(A(t)\right)$ is the resolvend of $A(t)$ and $R \big(\lambda; A(t_j)\big)$ is the resolvent operator of $A(t_j)$ 
for every finite sequence $0 \leq t_1 \leq t_2 \leq \ldots \leq t_k \leq T$. 
In this case, $A(t) \in \mathbb{G}\left(X,\,M,\,\beta\right)$.
\end{definition}

\begin{definition}\label{accretive}
Let $X$ be the Hilbert space. 	An operator $A$ in $X$ is considered {\it accretive} ($-A$ 
	is  {\it dissipative}) if:
	$$
	\mathrm{Re} \langle A x,\,x\rangle \geq 0, \,\,\,
	\mbox{for all}\;\;x \in D(A).
	$$
\end{definition}

It can be shown that (see \cite{Kato2, Reed2}) $A \in \mathbb{G}\left(X,\,1,\,\beta\right)$ 
if and only if
\begin{itemize}
	\item [(i)] $\langle A \phi, \,\phi \rangle \geq - \beta \,\|\phi\|^2,\;\;\forall \phi \in D(A)$. 
	\item [(ii)] $(A + \lambda)$ is onto for some (and therefore all) $\lambda > \beta$.
\end{itemize}
Here, $A$ is called a quasi-maximally accretive or quasi m-accretive.

As a consequence of (i), we have the following lemma. 
\begin{lemma}
A family $
\left\{ A(t) \colon t \in [0,\,T]\right\}$ of infinitesimal
generators of $C_0$ semigroups on the Hilbert space $X$ is stable if $A(t)$ is quasi m-accretive for all $t\in [0, T]$; that is, $A(t) \in \mathbb{G}\left(X,\,1,\,\beta\right)$.
\end{lemma}

A complete study of the evolution equation (\ref{evo11}) is in \cite{Kato5}, where the term CD-{\it system} was introduced according to the following definition:
   
\begin{definition}\label{CD}
	A triplet $\left\{\mathcal{A};\,X,\,Y\right\}$ having a one-parameter family of operators,
	$\mathcal{A} = \left\{ A(t) \colon  t \in I = [0, T]\right\}$ and a pair of real separable Banach spaces $Y \subset X$ is called a CD-{\it system} if the following conditions are satisfied:
	
	\begin{itemize}
		\item [(i)] $\mathcal{A} = \left\{ A(t) \colon t \in I\right\}$ is a stable family of (negative) 
		generators of $C_0$-semigroups on $X$, with stability constants $M$, $\beta$, i.e., 
		$A(t) \in \mathbb{G}\left(X,\,M,\,\beta\right)$.
		\item [(ii)] The domain $D\left(A(t)\right) =Y$ of $A(t)$ is independent of $t$. We regard $Y$ as Banach space, embedded continuously and densely in $X$.
		\item [(iii)] $A \in \mathrm{L}^{\infty}_*\left(I,\,\mathcal{B}\left(Y,\,X\right)\right)$ (or, equivalently, $\frac{dA}{dt} \in \mathrm{Lip}_*\left(I,\,\mathcal{B}\left(Y,\,X\right)\right)$),
		there exists $G$ such that $G(t) \in \mathcal{B}\left(Y,\,X\right)$ is defined
		for a.e. $t \in I$, is strongly measurable; then, $\|G(t)\|_{\mathcal{B}\left(Y,\,X\right)}$ is essentially bounded in $t$ such that
		$$
		A(t) y = A(0) y + \int_0^t\,G(s) y \,ds,\;\mathrm{for}\;\;y\in Y.
		$$
	\end{itemize}
\end{definition}
Using the hypotheses of a CD-{\it system}, Kato \cite{Kato4} constructed an evolution operator associated
to the family $\mathcal{A} = \left\{ A(t) \colon t \in I\right\}$ according to the following lemma:

\begin{lemma} \label{evol-operator}
Let $\left\{\mathcal{A};\,X,\,Y\right\}$ be a CD-{\it system} as in Definition~\ref{CD}. Then, an evolution operator $U(t,\,s)$ defined in the triangular domain $\triangle: 0 \leq s \leq t \leq T$ exists that satisfies the following properties:
\begin{itemize}
	\item [(i)] $U \in C_*\left(\triangle,\,\mathcal{B}(X)\right) \cap  C_*\left(\triangle,\,\mathcal{B}(Y)\right)$, where $ C_*$ denotes the space of all strongly continuous operator functions.
	\item [(ii)] $U(t,\,s) U(s,\,r) = U(t,\,r), \;\;U(s,\,s)= I\;\; (\mathrm{Identity\,\, operator})$.
	\item [(iii)] 
	$\partial_t U(t,\,s) y = - A(t) U(t,\,s) y,  \\
	\partial_s U(t,\,s)y  = U(t,\,s) A(s) y,\;\;\;y \in Y.$
\end{itemize}
\end{lemma}
The existence and uniqueness of the solution for \eqref{evo11} were also proved in \cite{Kato3}, according to the following theorem: 
\begin{theorem} \label{TeoKato1}
	Let $\left\{A,\,X,\,Y\right\}$ be the CD-{\it system}. Let $\phi \in Y$
	and $f \in \mathrm{Lip}(I,\,X)$. Then, problem (\ref{evo11}) has a unique solution $u \in C\left(I,\,Y\right) \cap C^ 1\left(I,\,X\right)$ that can be explicitly given by
	$$
	u(t) = U(t,0)\,\phi + \int_0^t U(t,\,\tau)\,f(\tau)\,d\tau.
	$$
\end{theorem}

\section{Local solution for the combustion problem}\label{combustion-local}
This section considers problem~\eqref{combustion-problem} for an unknown 
$u = (u_1,$ $\ldots, u_n) : \Omega_T \rightarrow \R^n $, where $\Omega_T = \R \times [0, T]$ for 
any fixed $T > 0$. We also consider the operator $L_i(t)$, defined in \eqref{operatorLi}, with domain in $C_0^{\infty}(\R)$; that is,
\begin{equation} \label{operatorLi-2}
L_i(t)\,v= -\alpha_i(t)\,v_{xx} + \beta_i(t)\, v_x, \,\,v \in C_0^\infty(\R), \,\, i = 1, \ldots, n,
\end{equation}
where $\alpha_i$ and $\beta_i$ are defined as in \eqref{alpha-beta} with the functions
$a_i, \,\,b_i, \,\,c_i, \,\,\lambda_i: \, \R \rightarrow \R$, and $y_i : \Omega_T \rightarrow \R$
all given. The function $d_i : \, \R \rightarrow \R$ is also given. 

Here, we propose the following hypotheses:
\begin{itemize}\label{Hy6Hy7}
\item[] {\bf(Hy6):} The functions $a_i$, $b_i$, $d_i$, $\lambda_i$, are twice differentiable, 
	and $c_i$ is three times differentiable and satisfy
	\begin{itemize}
		\item [(i)]  $k_1 \leq a_i(x), \, \lambda_i(x) \leq k_2$,  $0 \leq b_i(x),\, c_i(x) \leq k_2$, where $k_1 < k_2$ are positive constants. 
		\item [(ii)]  $a_i^{(k)}$, $b_i^{(k)}$, $d_i^{(k)}$, $\lambda_i^{(k)}$ for $k= 0, \, 1, \, 2$, and $c_i^{(k)}$ for $k = 0, \, 1, \,2, \, 3$ are in $\mathtt{L}^\infty(\R)$.
	\end{itemize}
	Note that, if $a_i$, $b_i$, $c_i$, $d_i$, $\lambda_i$ are constants, then all of these hypotheses are satisfied.
\item[]{\bf(Hy7):} The function $y_i$ is nonnegative and satisfies the following: 
\begin{itemize}
\item [(i)]	It is twice differentiable in $x$ and differentiable in $t$ for all 
$(x,\,t) \in \Omega_T$, $y_i$, ${(y_i)}_x$, ${(y_i)}_{xx}$, and $ {(y_i)}_t$ are in $\mathtt{L}^\infty(\Omega_T)$, and $\|y_i\|_{L^\infty(\Omega_T)} \leq k_3$, where $k_3$ is a positive constant.  
	\item [(ii)] $t \mapsto {(y_i)}_t$ is integrable on $[0,\,T]$ for all $x \in \R$.
	\item [(iii)] ${(y_i)}_t$ is twice differentiable in $x$, $ {(y_i)}_{tx} \in \mathtt{L}^\infty(\Omega_T)$ and $x \mapsto (y_i)_{txx} \in \mathtt{L}^2(\R)$ for all $t \in [0,\,T]$.
	\end{itemize}
\end{itemize}

\begin{lemma}\label{H1H2} 
If the operator $L_i(t)$ with the domain in $C_0^\infty(\R)$ satisfies
Hy6 and Hy7, it also satisfies Hy1 and Hy2 of \ref{scalar-local}. 
\end{lemma}
\begin{proof}
According to Hy7 (i), we have  $\|y_i\|_{\mathtt{L}^\infty(\Omega_T)} < k_3$, which implies that
	\begin{equation}\label{unifor}
	\dfrac{k_1}{k_2\,(1 + k_3)} < \alpha_i(x,\,t) <  \dfrac{k_2}{k_1},\,\,\mbox{for all}\,\,\, (x,\,t)\in \Omega_T.
	\end{equation}
	Thus, Hy1 is satisfied with $\mu_0 =  \dfrac{k_1}{k_2\,(1 + k_3)}$ and $\mu_1 =  \dfrac{k_2}{k_1}$.
	Hy2 (i) follows Hy6. To prove Hy2 (ii), we define
	\begin{equation*} 
	\widetilde{\alpha}_i(x,\,t) = \dfrac{-\lambda_i(x)\,b_i(x)\,(y_i)_t(x,\,t)}{\left(a_i(x)
		+ b_i(x)\,y_i(x,\,t)\right)^2}\;\;\mathrm{and}\;\;\widetilde{\beta}_i(x,\,t)
	= \dfrac{-c_i(x)\,b_i(x)\,(y_i)_t(x,\,t)}{\left(a_i(x) + b_i(x)\,y_i(x,\,t)\right)^2}.
	\end{equation*}
	Thus,
	\begin{eqnarray*}
		\alpha_i(x,\,t) & = & \alpha_i(x,\,0) + \int_0^t\,\widetilde{\alpha}_i(x,\,s)\,\mathrm{ds},  \\
		\beta_i(x,\,t)   &=& \beta_i(x,\,0) + \int_0^t\,\widetilde{\beta}_i(x,\,s)\,\mathrm{ds}.
	\end{eqnarray*}
	Therefore, Hy2 (ii) follows Hy7 (ii), completing the proof.
\end{proof}

The $L_i(t)$ operator can be extended to $\mathtt{H}^2(\R)$ according to the following lemma. 

\begin{lemma}\label{propagadorfront}
	Let Hy6 and Hy7 be satisfied.
	Then, for each $t\in [0,\,T]$ the $L_i(t)$ operator defined in $C_0^{\infty}(\R)$ is closable in $\mathtt{L}^2(\R)$.  
	and its closure, denoted by $A_i(t)$, has a domain independent of $t$ with $D\left(A_i(t)\right)= \mathtt{H}^2(\R)$. Furthermore, 
	${A_i(t)}$ is quasi m-accretive in $\mathtt{L}^2(\R)$ for each $t\in [0,\,T]$.
	That is, $A_i(t) \in G\left(\mathtt{L}^2(\R),\,1,\,\beta\right)$.
\end{lemma}
\begin{proof}
The proof is a consequence of Lemmas ~\ref{H1H2}, ~\ref{closure}, and ~\ref{acretivo}.
\end{proof}
Substituting $L_i(t) $ by its closure, the problem~\eqref{combustion-problem} can be written as follows:
\begin{equation}\label{combustion-problem2}
\left\{
\begin{array}{l}
(u_i)_t + A_i(t)u_i = f_i(t, u), \,\,\, 0 < t \leq T , \,\,\, i = 1,\ldots, n, \\
u_i(0) = \phi_i,\\
u = (u_1, \ldots, u_n),
\end{array}
\right.
\end{equation}
where
\begin{equation}\label{operatorAi}
A_i(t)\psi= -\alpha_i(t)\,\psi^{\prime\prime} + \beta_i(t)\,\psi^\prime,\;\;\psi \in D\left({A_i}(t)\right)= \mathtt{H}^2(\R),
\end{equation}
with $\alpha_i$ and $\beta_i$ as defined in \eqref{alpha-beta}.
Here, $\psi^{\prime\prime}$ and $\psi^\prime$ are distributions.

The next lemma states the existence of a family of semigroup evolution operators associated with 
problem~\eqref{combustion-problem2}. Following \cite{Kato6}, the proof uses the 
Lemma~\ref{propagadorfront}. 
Because it is similar to the proof given in Lemma~\ref{propagador}, we omit it here.

\begin{lemma}\label{propagador-combustao}
	We assume that Hy6 and Hy7 are satisfied. A unique family of semigroup 
	evolution operator $U_i(t, t')$ associated with $A_i$ exists, defined in a triangular domain as follows:
	$$
	(t,\,t^\prime) \in \triangle = \left\{(t,\,t^\prime) \in \R^2 \colon 0 \leq t^\prime \leq t \leq.
	T\right\}\longmapsto U_i(t,\,t^\prime) \in \mathcal{B}\left(\mathtt{L}^2(\R)\right).
	$$
that satisfy properties (i)-–(v), as in Lemma~\ref{propagador}.
\end{lemma}

Note that, if $f_i : \Omega_T \rightarrow \R$ is a known function that does not depend on $u$, then problem \eqref{combustion-problem2} became linear as
\begin{equation}\label{linear-problem}
\left\{
\begin{array}{l}
(u_i)_t + A_i(t)u_i = f_i(t), \,\,\,t >0,  \,\,\,  i = 1,\ldots, n,  \\
u_1(0) = \phi_i.
\end{array}
\right.
\end{equation}
Problem \eqref{linear-problem} is decoupled. Thus, the following theorem is a consequence of 
Theorem~\ref{TeoKato1} and Lemma \ref{propagadorfront}, see \cite{Kato6}:

\begin{theorem}\label{frontlinear} Let the same hypotheses as in Lemma \ref{propagadorfront} be satisfied.
	If $\phi_i \in \mathtt{H}^2(\R)$ and $f_i \in \mathrm{Lip}([0,\,T],\,\mathtt{L}^2(\R))$, for $i = 1, \ldots, n$, then the initial value problem
	(\ref{linear-problem}) has a unique solution:
	$$u = (u_1,\ldots, u_n) \in C\left([0,\,T],\,\mathtt{H}^2(\R)^n \right)
	\cap C^1\left([0,\,T],\,\mathtt{L}^2(\R)^n \right).$$
\end{theorem}

Assuming Hy6 and Hy7, we can define the following positive constants, which are used in the next lemma.
\begin{align}
&R_i:=\max_{k=0,1,2} \Big\{\|a_i^{(k)}\|_{\mathtt{L}^\infty(\R)},\,\, \|b_i^{(k)}\|_{\mathtt{L}^\infty(\R)},\,\, \|c_i\|_{\mathtt{L}^\infty(\R)},\,\, \|c_i^{(k+1)}\|_{\mathtt{L}^\infty(\R)},\,\, \|d_i^{(k)}\|_{\mathtt{L}^\infty(\R)},\nonumber\\ &\phantom{-------}\|\partial_x^{k}y_i\|_{\mathtt{L}^\infty(\Omega_T)},\,\,\|g^{(k)}\|_{\mathtt{L}^\infty(\R)} \Big\}, \label{Ri} \\
&\tilde R:=\max \{R_1, \ldots, R_n\}. \phantom{------} \label{R}
\end{align}

To prove Theorem~\ref{combustion-local-theo}, the source function must satisfy the same properties of Hy3 as in \ref{scalar-local}, with the spaces $\mathtt{L}^2(\R)$ and $\mathtt{H}^2(\R)$ substituted with $\mathtt{L}^2(\R)^n$ and $\mathtt{H}^2(\R)^n$, respectively. 
This is performed in the following lemma:

\begin{lemma} \label{properties-fi} 
	Assume that Hy6 and Hy7 are satisfied and $W \subset \mathtt{H}^2(\R)^n$ is any open ball centered at the origin. Then, the source function $f = (f_1, \ldots, f_n):
	[0, T] \times W \to \mathtt{H}^2(\R)^n$, for any fixed $T >0$,  satisfies the following properties: 
	\begin{itemize} 
		\item[(i)]  $\|f(t, \,w)\|_{\mathtt{H}^2(\R)^n} \leq \mu $, for all $t \in [0,\,T]$,  and 
		$w = (w_1, \ldots, w_n) \in W$,
		where $\mu$ is a constant depending on $k_1$ given in $Hy6(i)$, $\tilde R$ defined in \eqref{R}, and the radius $\rho$ of the ball $W$.
		\item[(ii)] For each $w \in W$, the function $t \mapsto f(t, \,w)$ is continuous on $[0, T]$ to $\mathtt{L}^2(\R)^n$.
		\item[(iii)] For each $t \in [0,\,T]$, the function $w \mapsto f(t, \,w)$ is Lipschitz continuous
		in $\mathtt{L}^2(\R)^n$; that is,
		$$
		\| f(t, \,w) - f(t, \,v)\|_{\mathtt{L}^2(\R)^n} \leq \kappa\, \| w - v\|_{\mathtt{L}^2(\R)^n}, 
		$$ 
		for all  $w = (w_1, \ldots, w_n)\,\,\,\mbox{and}\,\,\, v = (v_1, \ldots, v_n) \in W$, where the Lipschitz constant $\kappa$ does not depend on $t$.
	\end{itemize}
\end{lemma}	
\begin{proof} 	
	Property $(ii)$ is a consequence of the continuity of $y_i$ in
	$\Omega_T$ and the definition of $f_i$ given in \eqref{raction-function}; we prove that  $(i)$ 
	and $(iii)$.
	\medskip
	
	The proof of $(i)$.	For any $w \in W $, we must prove that 
	\begin{equation}\label{estimate-fi}
	\|f_i(t, \,w)\|_{\mathtt{H}^2(\R)} \leq \mu, \;\;t \in [0,\,T],\;\; w\in W,
	\end{equation}
	where $\mu = \mu(\tilde R, k_1, \rho)$. 	
	Thus, we must estimate $f_i$ and $\partial_x^2 f_i$ in the $\mathtt{L}^2(\R)$-norm.
	
	The proof is performed for $f_1$. The proofs for $f_i$, $i = 2, \ldots, n-1$, and $f_n$ can be performed similarly. 
	
	Let be
	$\alpha(x,t)=\frac{K_1 b_1 y_1}{a_1+b_1 y_1},$ $\beta(x,t)=\frac{(c_1)_x}{a_1+b_1y_1},$ $\gamma(x,t)=\frac{d_1 y_1}{a_1+b_1 y_1},$ $\delta_1(x,t)=\frac{q_1}{a_1+b_1 y_1}$, and $\delta_2(x,t)=-\frac{\overline{q}_1}{a_1+b_1y_1}$. 
	
	To simplify this notation, we set $z = w_1$. Thus,
	\begin{align*}
	f_1(x,t,w)=&\gamma(x,t)g(z)+\alpha(x,t)g(z)z-\beta(x,t)z -\delta_1(x,t)(z-w_2)\\
	+&\delta_2(x,t)(z-u_e) := l_1 + l_2 + l_3 + l_4 + l_5,
	\end{align*}
	\begin{align*}
	\partial_x f_1 &  = \gamma_x g(z)+\gamma(x,t)g'(z)z_x +\alpha_x g(z)z
	+ \alpha g'(z)z_x z + \alpha g(z)z_x  \\
	&-\beta_x z -\beta z_x + \partial_x\delta_1(z-w_2)+\delta_1\partial_x(z-w_2)+\partial_x\delta_2(z-u_e)\\
	&+ \delta_2\partial_x(z-u_e) ,
	\end{align*}
	\begin{align}\label{f1xx}
	\partial_x^2 f_1&=\gamma_{xx} g(z)+2\alpha_x g'(z)z_x z+2\alpha_x g(z) z_x+\alpha_{xx}g(z)z+\alpha g''(z)(z_x)^2 z
	\nonumber \\
	&+\alpha g'(z)z_{xx}z+2\alpha g'(z)(z_x)^2 +\alpha g(z)z_{xx}-\beta_{xx}z-2\beta_x z_x \nonumber \\
	&+2\gamma_x g'(z)z_x+\gamma g''(z)(z_x)^2 +\gamma g'(z)z_{xx}+\partial_x^2\delta_1(z-w_2)\nonumber \\
	&+2(\partial_x\delta_1)\partial_x(z-w_2)+\delta_1\partial_x^2(z-w_2)+(\partial_x^2\delta_2)(z-u_e)
    \nonumber \\
	&+2(\partial_x\delta_2)(\partial_x(z-u_e))+\delta_2\partial_x^2(z-u_e)+\alpha_{xx} g(z)z-\beta z_{xx}
	\nonumber \\
	&:=h_1+\dots + h_{20}.
	\end{align}
	Now, we calculate the $L^{\infty}$-norm of the terms,
	$\alpha$, $\alpha_x$, $\alpha_{xx},$ $\beta,$ $\beta_x,$ $\beta_{xx},$ $\gamma,$ $\gamma_x,$ $\gamma_{xx},$ $\delta_1,$ $(\delta_1)_x,$ $(\delta_1)_{xx},$ $(\delta_2)$, $(\delta_2)_x$, and $(\delta_2)_{xx}$, we find that all these norms
	are less than a positive constant that depends on $\tilde R$ and $k_1$. For example,
	$$\|\alpha\|_{L^{\infty}(\R)} = \Big\|\frac{K_1 b_1 y_1}{a_1+b_1y_1}\Big\|_{L^{\infty}(\R)} \leq 
	\frac{K_1 \|b_1\|_{L^{\infty}(\R)} \|y_1\|_{L^{\infty}(\R)}}{k_1} \leq const(\tilde R, k_1).$$
	\begin{align*}
	\| l_1 \|_{\mathtt{L}^2(\R)} &=\Vert \gamma(x,t)g(z)\Vert_{\mathtt{L}^2}
	=\bigg(\int \vert\gamma(x,t)\vert^2 \vert g(z)\vert^2 dx\bigg)^{\frac{1}{2}}\\
	&\leq \Vert\gamma\Vert_{L^{\infty}}\bigg(\int \vert g(z)\vert^2 dx\bigg)^{\frac{1}{2}}
	\leq \Vert \gamma\Vert_{L^{\infty}}\bigg(\int \vert g'(\widetilde{z})z\vert^2 dx\bigg)^{\frac{1}{2}}\\
	&\leq \Vert \gamma\Vert_{L^{\infty}} \Vert g'\Vert_{L^{\infty}}\bigg(\int \vert z\vert^2 dx\bigg)^{\frac{1}{2}}
	=\Vert \gamma\Vert_{L^{\infty}} \Vert g'\Vert_{L^{\infty}} \Vert z\Vert_{\mathtt{L}^2}\\
	&\leq const(\tilde R, k_1)\tilde R\Vert z\Vert_{\mathtt{H}^2}<\infty, \quad \tilde z \in (0,\, z).
	\end{align*}
	$\| l_i \|_{\mathtt{L}^2(\R)}$ for $i = 2, \cdots 5$ can be estimated similarly. 
	Thus, $\| f_1 \|_{\mathtt{L}^2(\R)} \leq const(\tilde R,$ $k_1, \rho)$. In addition, based on similar calculations,
	$\| \partial^2_x f_1 \|_{\mathtt{L}^2(\R)} \leq const(\tilde R, k_1, \rho)$,
	
	Therefore, $\| f_1 \|_{\mathtt{H}^2(\R)} \leq const(\tilde R, k_1, \rho)$, completing the proof of (i).
	
	\medskip
	Proof of (iii):
	
	This is sufficient to prove that, for each 
	$t \in [0,\,T]$, the function $w \mapsto f_i(t, \,w)$, $i = 1, \ldots n$, is Lipschitz continuous; 
	that is,
	$$
	\| f_i(t, \,w) -  f_i(t, \,v)\|_{\mathtt{L}^2(\R)} \leq \kappa_i\,\| w - v\|_{\mathtt{L}^2(\R)^n}, 
	$$ 
	for all  $w$, $v \in W$,  
	where Lipschitz constant $\kappa_i$ is independent of $t$. In this case, $\kappa:= \max\{\kappa_i, \,\, i= 1, \ldots, n \}$, which depends on $k_1$ and $\tilde R$.
	
	Let 
	$\gamma(\sigma)=(1-\sigma)v+\sigma w,$ 
	for $0\leq \sigma\leq 1$. 
	Let $F_i (\sigma)=f_i(t, \gamma(\sigma))$ and $F(\sigma)=\big(F_1(\sigma),$ $\ldots, F_2(\sigma)\big)$ for $i = 1, \ldots, n$.  From the definition of $f_i$ in \eqref{raction-function}, a straightforward computation is that
	\begin{align*}
	\partial_{w_j} f_1=  \frac{1}{a_1 + b_1 y_1}\big[&\big( - (c_1)_x + K_1 b_1 y_1 g(w_1) + (K_1 b_1 w_1 + d_1)y_1 g'(w_1) \\
	& - q_1 - \bar q_1 \big) \delta_{1,j} + q_1\delta_{2,j} \big],\\
	\partial_{w_j} f_i=  \frac{1}{a_i + b_i y_i}\big[&\big( - (c_i)_x + K_i b_i y_i g(w_i) + (K_i b_i w_i + d_i)y_i g'(w_i) \\
	& - q_{i-1} - q_i \big) \delta_{i,j} + q_{i-1}\delta_{i-1,j} + q_i \delta_{i+1, j}\big],\,\, i = 2, \ldots, n-1,\\
	\partial_{w_j} f_n=  \frac{1}{a_n + b_n y_n}\big[&\big( - (c_n)_x + K_n b_n y_n g(w_n) + (K_n b_n w_n + d_n)y_n g'(w_n) \\
	& - q_{n-1} - \bar q_2 \big) \delta_{n,j} + q_{n-1}\delta_{n-1,j} \big]\,,
	\end{align*} 
	where $\delta_{i,j}$ denotes the Kronecker delta.
	
	Then, in view of $\gamma(\sigma)\in W$, using Sobolev's embedding together with the $f_i$ derivatives, we have   	
	$$|\nabla_w f_i (\gamma(\sigma))| \leq const,$$
	for all $\sigma \in [0,1]$, where  $\nabla_w f_i=\big(\partial_{w_1}f_i, \ldots, \partial_{w_n}f_i\big)$ and $const$ does not 
	depend on $t$.
	
	From the mean value theorem, we have
	\begin{equation*}
	\begin{split} 
	|f_i(t,w) & - f_i(t, v)| = |F_i(1)-F_i(0)| =|F_{i}'(\sigma_0)|
	\\
	& = | \nabla_w f_i (\gamma(\sigma_0))\cdot (w-v)| 
	\leq const\,|w-v|,
	\end{split}
	\end{equation*}
	where $\sigma_0 \in (0,1)$, 
	
	Therefore, 
	\begin{equation}\label{fikappa}
	\|{f_i(t, w)-f_i(t, v))}\|_{\mathtt{L}^2(\R)} \leq \kappa_i \|{w-v}\|_{\mathtt{L}^2(\R)^n}, \quad 
	i=1,\ldots, n,
	\end{equation}
	where $\kappa_i$ does not depend on $t$, concluding the proof. 	
\end{proof}

\medskip

\noindent {\bf Proof of Theorem \ref{combustion-local-theo}.}
Let $W \subset \mathtt{H}^2(\R)^n$ be an open ball, as defined in Lemma~\ref{properties-fi}, possibly enlarged such that $\phi \in W$. Now, let $U(t,\,t')=\big(U_1(t,\,t'),\ldots, U_n(t,\,t')\big)$, where $U_i(t, t')$ is the evolution propagator generated by operator $A_i(t)$.

	Defining 
	\begin{align*}
	\Phi v(t) = U(t,\,0)\phi + \int_0^t\,U(t,\,s)\,f(s,\,v(s))\,ds,
	\end{align*}
	the local solution is given as the unique Banach fixed point of the map $\Phi$ in the set
	\begin{align}\label{conjuntoE_T}
	E_T=\big\{v\in C([0, T],\mathtt{L}^2(\R)^n): & \sup_{[0,T]}\|v(t)\|_{\mathtt{H}^2(\R)^n}\leq R, 
	\nonumber \\ 
	& \mbox{and} \quad \sup_{[0,T]}\|v(t)-U(t,0)\phi\|_{\mathtt{L}^2(\R)^n} \leq M \big\},
	\end{align}
	where analogous to the proof of Theorem~\ref{scalar-local-theo} $R$ and $M$ are positive constants defined such that $\Phi$ is a contraction in $E_{T'}$ for some $0 < T' \leq T$. This theorem also shows
	the local solution $u = \Phi u \in C\big([0,\,T'],\,W \big) \cap C^1 \big( [0,\,T'],\,\mathtt{L}^2(\R)^n \big)$.
	This completes the proof. 
	$\Box$

\section{Global solution for the combustion problem}\label{combustion-global}
This section considers problem~\eqref{combustion-problem} for an unknown 
$u = (u_1,$ $\ldots, u_n) : \Omega \rightarrow \R^n $, where $\Omega = \R \times [0, \infty)$. 
The operator $L_i(t)$, defined in \eqref{operatorLi}, with the domain in $C_0^{\infty}(\R)$, is such that 

$\alpha_i$ and $\beta_i$ are defined as in \eqref{alpha-beta} with the functions
$a_i, \,\,b_i, \,\,c_i,\,\,d_i, \,\,\lambda_i: \, \R \rightarrow \R$, and $y_i : \Omega_T \rightarrow \R$ all given. The source function
$f_i: \Omega \times \R^n \rightarrow \R$ is defined as in \eqref{raction-function}.

Here, we propose the following hypothesis:
\begin{itemize}
	\item [] {\bf (Hy8):} Coefficients $a_i$, $b_i$, $c_i$, and $\lambda_i$
	satisfy the same hypotheses, as in Hy6, and function $y_i$ satisfies 
	Hy7, with $\Omega_T$ substituted by $\Omega$.
\end{itemize}

We note that Lemmas~\ref{H1H2} and \ref{propagadorfront} remain valid if we replace Hy6 and Hy7 with Hy8. Thus, from Lemma~\ref{propagador-combustao}, the evolution operator for this case also exists.

To prove Theorem~\ref{combustion-global-theo}, source function $f$ must satisfy the same properties as
Hy5 in \ref{scalar-global}, with the spaces $\mathtt{L}^2(\R)$ and $\mathtt{H}^2(\R)$ substituted with $\mathtt{L}^2(\R)^n$ and
$\mathtt{H}^2(\R)^n$. This is performed in the following lemma.

\begin{lemma} \label{properties-fi-global} 
	Assume that Hy8 is satisfied, and as in Lemma~\ref{properties-fi}, $W \subset \mathtt{H}^2(\R)^n$ is any open ball centered at the origin.
	Subsequently, the function $f = (f_1, \cdots, f_n)$ satisfies the following properties:
	\begin{itemize}
		\item[(i)]  $\|f(t, \,w)\|_{\mathtt{H}^2(\R)^n} \leq \mu $, for all $t \ge 0$ and $w \in W$,
		where $\mu$ is a constant depending on $k_1$ given by $Hy6(i)$, on $\tilde R$ the constant defined in \eqref{R}, and on $\rho$ the radius of the open ball $W$.
		\item[(ii)] For each $w \in W$, the function $t \mapsto f(t, \,w)$ is continuous on 
		$[0, \infty)$ to $\mathtt{L}^2(\R)^n$.
		\item[(iii)] For each $t \ge 0$, the function $w \mapsto f(t, \,w)$ is $\mathtt{L}^2(\R)^n$-Lipschitz continuous in $\mathtt{L}^2(\R)^n$; that is,
		$$
		\| f(t, \,w) - f(t, \,v)\|_{\mathtt{L}^2(\R)^n} \leq \kappa\,\| w - v\|_{\mathtt{L}^2(\R)^n}, 
		$$ 
		for all $w ,\, v \in W$, where the Lipschitz constant $\kappa$ does not depend on $t$.
	\end{itemize}
\end{lemma}	

\begin{proof}
Property $(ii)$ is a consequence of the continuity of $y_i$ in $\Omega$ and definition of $f_i$. The proofs of $(i)$ and $(iii)$ are analogous to those of properties (i) and (iii) in Lemma~\ref{properties-fi}, respectively, because they are valid for any $T \ge 0$.  
\end{proof}	

\noindent {\bf Proof of Theorem \ref{combustion-global-theo}.}
Let $W \subset \mathtt{H}^2(\R)^n$ be an open ball, as defined in Lemma~\ref{properties-fi-global}, possibly enlarged such that $\phi = (\phi_1, \ldots, \phi_n)  \in W$.

Now, we define 
	\begin{align} \label{uniqueT}
	\nonumber
	T^* = & \sup \big\{ T > 0, \,\,\, \mbox{such that the unique solution
	 to \eqref{combustion-problem2} in $[0,\,T]$ exists} \big\}.
	\end{align}
		
	Let $u=(u_1, \ldots, u_n)$ be the local solution to \eqref{combustion-problem2} given by 
	Theorem~\ref{combustion-local-theo} is defined in its maximal interval $[0, T^*)$.

We also defined the following linear scalar equation:
\begin{equation}\label{eq-AG}
\partial_t v+A_{G}(t)v=f_i(t,u), \,\,\, 0 < t < T^*, \,\,\,v \in D(A_{G}(t)) = \mathtt{H}^2(\R),
\end{equation}
where operator $A_G(t)$,
\begin{equation}\label{op-AG}
A_G(t) v = - \alpha_i(t) v_{xx} + \beta_i(t) v_x, \,\,\, t \ge 0, 
\end{equation}
represents the closure of $L_i(t)$.
If $v:=u_i$, for $i=1,\ldots, n$, \eqref{eq-AG} is equivalent to the differential equation 
in\eqref{combustion-problem2}.

Note that, by hypothesis Hy8, functions $\alpha_i$ and $\beta_i$ satisfy Hy4, and  by Lemma \ref{properties-fi-global}, $f_i$ satisfies Hy5, as in \ref{scalar-global}. 
Therefore, the rest of the proof can be given by following the proof of Theorem~\ref{scalar-global-theo} for the scalar equation; thus, we omit it here. 
$\Box$

\section{Continuous dependence for the combustion problem}\label{combustion-continuous-dependence}

This section addresses the proof of Theorem~\ref{combusiton-continuous-dependence-theo}. 
As referred to in Section~\ref{introduction}, we first prove the continuous dependence for only the initial data and then prove the continuous dependence for the initial data and parameters together. In each case, the proof holds for any fixed $T>0$; it encompasses local and global continuity.  

The function $w \to f(t, w)$ must be $\mathtt{H}^2(\R)^n$-Lipschitz continuous, which is proven in the following lemma.

\begin{lemma} \label{properties-fi-H2} 
	Assume that Hy6 and Hy7 are satisfied and $W \subset \mathtt{H}^2(\R)^n$ is any open ball centered at the origin. Then, the source function 
	$w \mapsto f(t,\,w)$ is $\mathtt{H}^2(\R)^n$-Lipschitz continuous; that is,
	$$
	\| f(t,\,w) -  f(t,\,v)\|_{\mathtt{H}^2(\R)^n} \leq \kappa\,
	\|w - v\|_{\mathtt{H}^2(\R)^n}, \;\; \forall \, w,\,v \in W,
	$$
	where the Lipschitz constant $\kappa$ does not depend on $t$. 
\end{lemma}

\begin{proof}
The proof is performed for $f_1$. The proofs for $f_i$, $i = 2, \ldots, n-1$, and $f_n$ can be performed similarly. Let $\alpha$, $\beta$, $\gamma$, $\delta_1$, and $\delta_2$ as at the beginning of the proof of Lemma \ref{properties-fi}. 	
Let also $w=(w_1,\cdots, w_n)$ and $v=(v_1,\cdots, v_n)$. Then, by writing $z=w_1$ and $\nu=v_1$, it follows that

\begin{align}\label{px2huv}
\partial_x^2\big(&f_1(t,w)-f_1(t,v)\big)=  \gamma_{xx} (g(z)-g(\nu))+2\alpha_x (g'(z)z_x z- g'(\nu)\nu_x \nu) \nonumber \\
& + 2\alpha_x (g(z)z_x-g(\nu)\nu_x)+\alpha(g''(z)z_x^2 z-g''(\nu)\nu_x^2 \nu)+\alpha_{xx}(g(z)z-g(\nu)\nu)
\nonumber \\
& + \alpha (g'(z)z_{xx}z-g'(\nu)\nu_{xx} \nu)+2\alpha (g'(z)z_x^2-g'(\nu)\nu_x^2) \nonumber \\
& + \alpha(g(z)z_{xx}-g(\nu)\nu_{xx}) -\beta_{xx} (z-\nu)-2\beta_x (z_x-\nu_x) 
-\beta(z_{xx}-\nu_{xx})\nonumber \\
& + 2\gamma_x (g'(z)z_x-g'(\nu)\nu_x)+\gamma(g''(z)z_x^2-g''(\nu)\nu_x^2) \nonumber \\
& + \gamma(g'(z)z_{xx} - g'(\nu)\nu_{xx})+\partial_x^2 \delta_{1}(z-\nu-w_2+v_2) \nonumber \\
& + \partial_x^2 \delta_2 (z-\nu) +\delta_2 \partial_x^2 (z-\nu)+\delta_1 \partial_x^2 (z-\nu-w_2+v_2) 
\nonumber \\
& + 2\partial_x \delta_1 \partial_x (z-\nu-w_2+v_2)+2\partial_x \delta_2 \partial_x (z-\nu) \nonumber \\ 
& :=  \ \tilde h_1+\cdots+\tilde h_{20}.
\end{align} 

To estimate $\tilde h_1$, we use the mean value inequality. Thus, 
\begin{align}\label{h1tilde}
\|\tilde h_1\|_{\mathtt{L}^2(\mathbb R)}&=\| \gamma_{xx} (g(z)-g(\nu))\|_{\mathtt{L}^2(\mathbb R)} \nonumber \\
&\leq \|\gamma_{xx}\|_{L^\infty(\Omega_T)}\|g'\|_{L^\infty(\R)}\|z-\nu\|_{\mathtt{L}^2(\mathbb R)} \nonumber \\
&\leq \tilde c_1(R,r_1)\|z-\nu\|_{\mathtt{L}^2(\mathbb R)}.
\end{align}

To estimate the other terms in \eqref{px2huv}, we can use Sobolev embedding and proceed in a manner similar to that above. Thus, we have 

\begin{equation}\label{tildehj}
\|\tilde h_j\|_{\mathtt{L}^2(\mathbb R)} \leq \tilde c_j(R,r_1)\|z-\nu\|_{\mathtt{H}^2(\R)}, \quad j=1,...,20.
\end{equation}

Therefore, by using \eqref{px2huv}, \eqref{h1tilde}, \eqref{tildehj}, and Lemma~\ref{properties-fi-global}(iii), we obtain the desired result.
\end{proof}
	
\medskip
\noindent {\bf Proof of Theorem \ref{combusiton-continuous-dependence-theo}.}
 (i) {\em Continuous dependence on the initial data.}
\medskip
  
We assume that $W \subset \mathtt{H}^2(\R)^n$ is an open ball, as defined in Lemma~\ref{properties-fi}, which can 
possibly enlarged such that $\phi = (\phi_1, \ldots, \phi_n)  \in W$. We assume also that  
$u = (u_1,\ldots, u_n)\in C([0, T],W^\prime)\cap C^1([0, T],\mathtt{L}^2(\R)^n)$ is the solution, 
given by Theorem~\ref{combustion-global-theo}:

Let $\phi^j$ be a sequence in $W$ such that $\phi^j \to \phi$ in the 
$\mathtt{H}^2(\R)^n$ norms. Let $u$ and 
$u^j$ be the solutions to problem~\eqref{combustion-problem2} with $u(0)=\phi$ and $u^j(0)=\phi^j$, respectively, as given by Theorem~\ref{combustion-global-theo}. We have to prove that $u^j \to u$ in $C([0, T],W^\prime)\cap C^1([0, T], \mathtt{L}^2(\R)^n)$, for any fixed $T>0$, 
and some $W^\prime\supset W$. We have
\begin{align*} \label{uni1}
&\|u^j(t) - u(t) \|_{\mathtt{H}^2(\R)^n}  \leq  \|U(t,\,0)(\phi^j - \phi)\|_{\mathtt{H}^2(\R)^n} + 
\Big\|\int_0^t U(t,\,s)\big(f(t, u^j(s))  \\
& - f(t, u(s))\big)\, d s \Big\|_{\mathtt{H}^2(\R)^n} \leq  e^{\tilde \beta T} \,\|\phi^j - \phi\|_{\mathtt{H}^2(\R)^n} +
e^{\tilde \beta T}\kappa\int_0^t\|u^j(s) - u(s)\|_{\mathtt{H}^2(\R)^n} d s .  
\end{align*}
Applying Gronwall's inequality, we have
	\begin{equation}
	\|u^j(t) - u(t) \|_{\mathtt{H}^2(\R)^n} \leq \tilde \kappa \|\phi^j - \phi\|_{\mathtt{H}^2(\R)^n} ,
	\end{equation}
where $\tilde \kappa$ is independent on $t$. 	By taking the supremum over $t\in[0,T]$, 
	\begin{equation} \label{conv1}
	\|u^j - u\|_{C([0, T], \mathtt{H}^2(\R)^n)} \leq \tilde \kappa \|\phi^j - \phi\|_{\mathtt{H}^2(\R)^n} .
	\end{equation}
Thus,
	\begin{equation*}
	\label{eq74}
	u^j\rightarrow u \,\,\text{in}\,\,C([0, T],W'),\,\,\text{since}\,\,\phi^j\rightarrow \phi\,\, 
	\text{in}\,\, W.
	\end{equation*}
Additionally, from \eqref{combustion-problem}, we have
\begin{align*}
u^j_t(t)-u_t(t) &= \alpha_i(t)(u^j_{xx}(t)-u_{xx}(t))-\beta_i(t)(u^j_x(t)-u_x(t))\\
&+f(t, u^j(t))-f(t, u(t)).
\end{align*}
Thus,
\begin{align*}  
\|u^j_t(t)&-u_t(t)\|_{\mathtt{L}^2(\R)^n} \leq \|\alpha_i(t)(u^j_{xx}(t)-u_{xx}(t))\|_{\mathtt{L}^2(\R)^n}\nonumber \\  
& + \|\beta_i(t)(u^j_x(t)-u_x(t))\|_{\mathtt{L}^2(\R)^n}  
 + \|f(t, u^j(t))-f(t, u(t))\|_{\mathtt{L}^2(\R)^n} \nonumber \\ 
& \leq const\,\|(u^j_{xx}(t)-u_{xx}(t))\|_{\mathtt{L}^2(\R)^n}  
+const\,\|(u^j_x(t)-u_x(t))\|_{\mathtt{L}^2(\R)^n} \nonumber \\
& +const\,\|u^j(t)-u(t)\|_{\mathtt{L}^2(\R)^n}
 \leq const\,\|u^j(t)-u(t)\|_{\mathtt{H}^2(\R)^n}\,. \nonumber
\end{align*}
Again, taking the supremum over $t\in[0,T]$, we obtain
\begin{equation}\label{conv2}
	\|u^j_t-u_t\|_{C([0, T],{\mathtt{L}^2(\R)^n})}\leq const\,\|u^j-u\|_{C([0, T],\mathtt{H}^2(\R)^n)}. 
\end{equation}
Therefore, from \eqref{conv1}, we obtain 
	
	$u^j\rightarrow u \,\,\text{in}\,\,C([0, T],W)\cap C^1([0, T],\mathtt{L}^2(\R)^n), \,\,
	\text{when}\,\,\phi^j\rightarrow \phi\,\, \text{in}\,\, W $.

\bigskip
\noindent (ii) {\em Continuous dependence with respect to initial data and parameters.} 
\medskip

Here, we assume that parameters $a_i$, $b_i$, $(c_i)_x$, $d_i$, $\lambda_i$ and function 
$y_i(t)$, where $t\in[0,T]$, are all in $\mathtt{H}^2(\R)$. 
\begin{lemma}\label{thF}
(i) Let $(a_i)^j$, $(b_i)^j$, $((c_i)_x)^j$, $(d_i)^j$, $(\lambda_i)^j$, and $(y_i)^j(t))$ be sequences in $\mathtt{H}^2(\R)$ such that
 $(a_i)^j\rightarrow a_i$, $(b_i)^j\rightarrow b_i$, $((c_i)_x)^j\rightarrow (c_i)_x$, $(d_i)^j\rightarrow d_i$, $(\lambda_i)^j\rightarrow \lambda_i$, and $(y_i)^j(t))\rightarrow y_i(t)$ in $\mathtt{H}^2(\R)$. Then,
 \begin{align} \label{convergence-fi}
(f_i)^j(t, w)\rightarrow f_i(t,w)\,\,\, \mbox{on}\,\,\ \mathtt{H}^2(\R), \,\,\, 
\mbox{for each}\,\,\, t \in [0, T],
 \end{align}  
where $f_i$ is defined as 
\eqref{raction-function}, $(f_i)^j$ is defined in \eqref{raction-function-j} below, and 
$w = (w_1, \ldots, w_n)$ $\in \mathtt{H}^2(\R)^n$.

\noindent (ii) If, in the hypotheses of item (i), the convergences are in $L^\infty(\mathbb R)$, then 
$$
(\alpha_i)^j(t)=\dfrac{(\lambda_i)^j}{(a_i)^j+(b_i)^j(y_i)^j(t)}\to \dfrac{\lambda_i}{a_i+b_i y_i(t)}= \alpha_i(t),
$$
and
$$
(\beta_i)^j(t)=\dfrac{(c_i)^j}{(a_i)^j+(b_i)^j(y_i)^j(t)}\to \dfrac{c_i}{a_i+b_i y_i(t)}= \beta_i(t)
$$
for $L^\infty(\mathbb R)$.
\end{lemma}
\begin{proof}
From \eqref{raction-function}, we define
\begin{align} \label{raction-function-j} 
&(f_1)^j(t, w) = \frac{-((c_1)^j)_x\,w_1 }{(a_1)^j + (b_1)^j (y_1)^j(t)} + 
\frac{\big(K_1 (b_1)^j w_1+(d_1)^j\big)(y_1)^j(t)\,g(w_1)}{(a_1)^j + (b_1)^j (y_1)^j(t)} 
\nonumber \\
& \phantom{------}+ \frac{q_1(w_2 - w_1)}{(a_1)^j + (b_1)^j (y_1)^j(t)}
- \frac{\overline q_1(w_1 - u_e)}{(a_1)^j + (b_1)^j (y_1)^j(t)}\,,
\nonumber \\
\nonumber \\
&(f_i)^j(t, w) = \frac{-((c_i)_x)^j\,w_i }{(a_i)^j + (b_i)^j (y_i)^j(t)} + 
\frac{\big(K_i (b_i)^j w_i+(d_i)^j\big)(y_i)^j(t)\,g(w_i)}{(a_i)^j + (b_i)^j (y_i)^j(t)} 
\nonumber \\
& \phantom{------}+ \frac{q_i(w_{i+1} - w_i)}{(a_i)^j + (b_i)^j (y_i)^j(t)}
- \frac{q_{i-1}(w_i - w_{i-1})}{(a_i)^j + (b_i)^j (y_i)^j(t)}\,, \,\,i = 2,\ldots, n-1,
 \\
&(f_n)^j(t, w) = \frac{-((c_n)^j)_x\,w_n }{(a_n)^j + (b_n)^j (y_n)^j(t)} + 
\frac{\big(K_n (b_n)^j w_n+(d_n)^j\big)(y_n)^j(t)\,g(w_n)}{(a_n)^j + (b_n)^j (y_n)^j(t)} 
\nonumber \\
& \phantom{------} - \frac{q_{n-1}(w_n - w_{n-1})}{(a_n)^j + (b_n)^j (y_n)^j(t)}
- \frac{\overline q_2(w_n - u_e)}{(a_n)^j + (b_n)^j (y_n)^j(t)}\,,
\nonumber 
\end{align}

The proof of (i) is an easy consequence of the following: if $h_1$, $(h_1)^j$, $h_2$, and $(h_2)^j$ are functions in $\mathtt{H}^2(\R)$ such that 
$|h_2|,\,\,|(h_2)^j| \ge \mbox{const} > 0$, $(h_1)^j \to h_1$, and $(h_2)^j \to h_2$ on $\mathtt{H}^2(\R)$, then $\dfrac{(h_1)^j}{(h_2)^j}\rightarrow \dfrac{h_1}{h_2}$ on $\mathtt{H}^2(\R)$. 

The proof of (ii) is similar to that of (i) by substituting $\mathtt{H}^2(\R)$  by $L^\infty(\R)$. 	
\end{proof}

\begin{lemma}\label{thA}
We consider problem~\eqref{combustion-problem2} with the differential operator $A_i$ defined in 
\eqref{operatorAi}.
Let $\alpha_i(t), \,\,\beta_i(t),\,\,(\alpha_i)^j(t), \,\,(\beta_i)^j(t)\in \mathtt{H}^2(\R)$ as in Lemma~\ref{thF}. Let also the 
sequence of operators 
$$(A_i)^j(t)u = (\alpha_i)^j(t)u_{xx}+(\beta_i)^j(t)u_{x}.$$	
Then,
\begin{enumerate}
\item [(i)] $(A_i)^j(t)\rightarrow A_i(t)$ on $\mathcal{B}\big(\mathtt{H}^2(\R),\,\mathtt{L}^2(\R)\big)$, 
and
\item [(ii)] $\lim_{|E|\rightarrow 0}\int_E \|A^j(t)\|_{\mathcal{B}\big(\mathtt{H}^2(\R),\,\mathtt{L}^2(\R)\big)}dt\rightarrow 0$. 
\end{enumerate}
\end{lemma}
\begin{proof}
We obtain	
	\begin{align*}
\|(A_i)^j&(t)u-A_i(t)u\|_{\mathtt{L}^2(\R)} \\
&= \|\big((\alpha_i)^j(t) -\alpha_i(t)\big)u_{xx}+ \big((\beta_i)^j(t)-\beta_i(t)\big)u_{x}\|_{\mathtt{L}^2(\R)}\\
&\leq \|\big((\alpha_i)^j(t)-\alpha_i(t)\big)u_{xx}\|_{\mathtt{L}^2(\R)}+
\|\big((\beta_i)^j(t)-\beta_i(t)\big)u_{x}\|_{\mathtt{L}^2(\R)}\\
&\leq  \sup_\R\big(|(\alpha_i)^j(t)-\alpha_i(t)|^2+|(\beta_i)^j(t)-\beta_i(t)|^2 \big)\|u\|_{\mathtt{H}^2(\R)}.
	\end{align*} 
	Therefore,
	\begin{eqnarray*}
	\|A^j(t)-A(t)\|_{\mathcal{B}(\mathtt{H}^2,\,\mathtt{L}^2)}&\leq& \|(\alpha_i)^j( t)-\alpha(t)\|^2_{L^\infty(\R)}+\|(\beta_i)^j-\beta(t)\|^2_{L^\infty(\R)},
	\end{eqnarray*}
indicating that $A^j(t)\rightarrow A(t)$ on $\mathcal{B}\big(\mathtt{H}^2(\R),\,\mathtt{L}^2(\R)\big)$, when $j \to \infty$.
	
The proof of (ii) follows from
	
\begin{align*}
\|A^j(t)u\|_{\mathtt{L}^2(\R)}&\leq \|(\alpha_i)^j(t)u_{xx}\|_{\mathtt{L}^2(\R)} +
\|(\beta_i)^j(t) u_{x}\|_{\mathtt{L}^2(\R)} \\
& \leq const\, (\|u_{xx}\|_{\mathtt{L}^2(\R)} 
+\|u_{x}\|_{\mathtt{L}^2(\R)}\leq const\, \|u\|_{\mathtt{H}^2(\R)},
\end{align*}
indicating that $\|A^j(t)\|_{\mathcal{B}(\mathtt{H}^2(\R),\,\mathtt{L}^2(\R))}\leq const\, $. Therefore, 
	$$\lim_{|E|\rightarrow 0}\int_E \|A^j(t)\|_{\mathcal{B}(\mathtt{H}^2(\R),\,\mathtt{L}^2(\R))}dt\rightarrow 0. $$
\end{proof}

Now, consider the following problems:

\begin{equation}\label{solu1}
\left\{\begin{array}{l}
\partial_t u+A(t)u=f(t,u),\\
u(0)=\phi,\,\,
\end{array}\right.
\end{equation}
and
\begin{equation}\label{solu2}
\left\{\begin{array}{l}
\partial_t u+A^j(t)u=f^j(t,u),\\
u^j(0)=\phi^j,\,\,
\end{array}\right.
\end{equation}
where $A(t) = \big(A_1(t), \ldots A_n(t) \big)$, 
$A^j(t) = \big((A_1)^j(t), \ldots (A_n)^j(t) \big)$, and $\phi^j$ is a sequence in $W$. 

The proof of the continuous dependence for the initial data and parameters is a consequence of the following theorem.
\begin{theorem}
As in Lemma~\ref{thF}, let 
$(a_i)^j\rightarrow a_i$, $(b_i)^j\rightarrow b_i$, $((c_i)_x)^j\rightarrow (c_i)_x$, $(d_i)^j\rightarrow d_i$, $(\lambda_i)^j\rightarrow \lambda_i$, and $(y_i)^j(t))\rightarrow y_i(t)$  be on $\mathtt{H}^2(\R)$.	Let also $\phi^j \to \phi$. If
$u(t),u^j(t)\in C([0, T],W)\cap C^1([0, T],\mathtt{L}^2(\R))$ are solutions to \eqref{solu1} and \eqref{solu2}, respectively, then $u^j\rightarrow u$ in 
$ C([0, T],W)\cap C^1([0, T],\mathtt{L}^2(\R))$,
\end{theorem}

\begin{proof}
From Lemma~\ref{thF}, $f^j(t, u(t)) \rightarrow f(t, u(t))$ in $\mathtt{H}^2(\R)$ and
$\alpha_i^j(t) \to \alpha_i(t)$ in $\mathtt{L}^{\infty}$. From Lemma~\ref{thA}, we obtain  	
$A^j(t)\rightarrow A(t)$ in $\mathcal{B}\big(\mathtt{H}^2(\R),\,\mathtt{L}^2(\R)\big)$. The remaining proof is provided by \cite[Theorem~7]{Kato3}. 	
\end{proof}	

\section{Concluding remarks}

This study used the semigroup theory to prove the suitability of an initial value problem, which models a combustion front through a multi-layer porous medium. The following are some  
advantages of using the semigroups theory of operators and Kato's theory in the proofs:
i) They allowed us to consider a more realistic model, where some physical parameters (e.g., porosity, thermal conductivity, and initial fuel concentration) depend on the spatial variable $x$. Earlier studies \cite{Batista1} and
\cite{Batista2} considered that these parameters were constant.
ii) The uniqueness of the solution is a natural consequence of the uniqueness of the evolution operator both locally and globally.
iii) The proof of continuous dependence for the initial data and parameters was possible because the solution has an explicit representation given by the integral equation \eqref{integral-form1},
in the case also of both the local and global in time solutions. 
 
We remarked that the semigroup theory does not apply directly to solve the corresponding problem 
for the complete system when the fuel concentrations are also unknown because the equations for
concentrations are not of the reaction-diffusion type (see the differential equations in 
\cite[equation~(2)]{Batista2}). However, we can use the solution of the simplified problem depending on the concentrations obtained in this study to build an iteration scheme that converges to the solution of the corresponding complete problem. This is an interesting problem to be investigated.

\begin{appendix}

\section{Local solution for a scalar equation} \label{scalar-local}

Here, we describe the main results of the general semilinear second-order scalar equations. 
The following second-order semilinear scalar equation is considered:  
\begin{equation}\label{scalar-local-eq}
\partial_t u - a(x,\, t)\, u_{xx} + b(x,\,t)\,u_x = f(x,\,t,\,u),\,\,x \in \R, \,\, t \in (0, T],
\end{equation}
where $a,\,\, b\, :\Omega_T = \R \times [0, T] \to \R$ and $f : \Omega_T \times \R \to \R$ are the functions for some $T >0$. We also consider the family of associated differential operators, defined as 
\begin{equation}\label{ope-L}
\left(L(t)\phi\right)(x) = -a(x,t) \phi^{\prime\prime}(x) + b(x,t) \phi^\prime(x),\;\;
\phi \in C^{\infty}_0(\R),\,\,\,t \in [0, T].
\end{equation}

To solve \eqref{scalar-local-eq} using semigroup theory, the following hypothesis should be established:
first for the coefficients of $L(t)$ and then for the source function $f$ as follows: 
 
\begin{itemize}
	\item [] {\bf(Hy1):} $L(t)$ is uniformly parabolic in $\Omega_T$; that is, there are positive constants $\mu_0$ and
	$\mu_1$ such that
	\begin{equation}\label{A1}
	\mu_0\leq a(x,t)\leq \mu_1, \quad \text{for all} \,\,\, (x,t)\in\Omega_T,
	\end{equation}
	
	\item [] {\bf (Hy2)}: The coefficients of $L(t)$ satisfies
	\begin{itemize}
		\item [(i)] $a(x,\,t)$ is twice differentiable in $x$, $b(x,\,t)$ is differentiable in $x$, for all
		$(x,\,t) \in \Omega_T$, and 
		\begin{equation}\label{A2}
		a,\, b, \,a_x, \,b_x,\, \,a_{xx} \in \mathtt{L}^\infty(\Omega_T).
		\end{equation}
		\item [(ii)] $a(x,\,\cdot)$ and $b(x,\,\cdot)$ are a measurable function on $[0,\,T]$, and  there are measurable
		functions  $\tilde{a},\;\tilde{b} \colon \Omega_T \rightarrow \R$, integrable on $[0,\,T]$ and
		$\tilde{a}(\cdot,\,t), \,\tilde{b}(\cdot,\,t) \in \mathtt{L}^\infty(\R)$ uniformly in $t$ such that $a$ and $b$ 
		are equal
		to the indefinite integral of $\tilde{a}$ and $\tilde{b}$, respectively; that is,
		\begin{eqnarray*}
			a(x,\,t) & = & a(x,\,0) + \int_0^t\,\tilde{a}(x,\,s)\,\mathrm{ds},  \\
			b(x,\,t)   &=& b(x,\,0) + \int_0^t\,\tilde{b}(x,\,s)\,\mathrm{ds}.
		\end{eqnarray*}
	\end{itemize}
\end{itemize}

\begin{itemize}
	\item [] {\bf (Hy3):}  The function $(t,\,w) \mapsto f(t, \,w)$, satisfies the following:
	\begin{itemize}
		\item[(i)] $f$ is bounded  on $[0,\,T] \times W$ to $\mathtt{H}^2(\R)$, where $W$ is any open ball in $\mathtt{H}^2(\R)$. That is,
		$$
		\|f(t, \,w)\|_{\mathtt{H}^2(\R)} \leq \mu, \,\,\mbox{for}\,\,t \in [0,\,T]\,\,\mbox{and}\,\, w\in W,
		$$
		being $\mu$ a constant depending only on the radius of $W$.
		\item[(ii)] For each $w \in W$, the function $t \mapsto f(t, \,w)$ is continuous on $[0,\,T]$ to $\mathtt{L}^2(\R)$.
		\item[(iii)] For each $t \in [0,\,T]$, $w \mapsto f(t, \,w)$ is $\mathtt{L}^2(\R)$-Lipschitz continuous in $W$; that is,
		$$
		\| f(t, \,w) -  f(t, \,v)\|_{\mathtt{L}^2(\R)} \leq \kappa\,\| w - v\|_{\mathtt{L}^2(\R)}, 
		$$
		where $\kappa$ is a constant depending only on the radius of $W$.
	\end{itemize}
\end{itemize}

\begin{lemma}\label{closure}
	If $\left\{L(t)\right\}_{ t\in [0,\,T]}$ satisfies Hy1 and Hy2, 
	then the operator $L(t)$  is closable in $\mathtt{L}^2(\R)$. Furthermore, the closure of $L(t)$, denoted by $A(t)$, is independent of $t$ with $D\left(A(t)\right)= \mathtt{H}^2(\R)$. 
\end{lemma}

\begin{proof}
We put 
	$$
	L(t) \phi = B_0(t) \phi + B_1(t) \phi,
	$$
	where 
	$$
	B_0(t) \phi = - a(\cdot,\,t)\,\phi^{\prime\prime},\,\,\mbox{and}\,\, 
	B_1(t) \phi =  b(\cdot,\,t)\,\phi^{\prime}, \,\,\, \mbox{for} \,\, \phi \in C^\infty_0(\R),
	$$

	Considering conditions (\ref{A1}) and (\ref{A2}), $B_0(t)$ and $B_1(t)$  are closable in $\mathtt{L}^2(\R)$, the closures
	$A_0(t)$ of $B_0(t)$ and $A_1(t)$ of $B_1(t)$ are independent of $t$ given by
	$$
	D\left(A_0(t)\right)=\left\{ \phi \in \mathtt{L}^2(\R) \colon \phi^{\prime\prime} \in \mathtt{L}^2(\R)\right\} = \mathtt{H}^2(\R),
	$$
	and
	$$
	D\left(A_1(t)\right)=\left\{ \phi \in \mathtt{L}^2(\R) \colon \phi^{\prime} \in \mathtt{L}^2(\R)\right\} = \mathtt{H}^1(\R).
	$$
	We can write
	$$ 
	A_0(t) \phi = -a(\cdot,\,t)\,\phi^{\prime \prime}, \,\, \phi \in \mathtt{H}^2(\R), 
	\,\, \mbox{and} \,\,\,
	A_1(t) \phi = b(\cdot,\,t)\,\phi^{\prime}, \, \, \phi \in \mathtt{H}^1(\R),
	$$
	where $\phi^\prime$ and $\phi^{\prime\prime}$ are in the distribution sense.
	
	We now show that $A_1(t)$ is relatively bounded with respect to $A_0(t)$. In fact, since
	$D\left(A_0(t)\right)\subset D\left(A_1(t)\right)$, if $\phi \in \mathtt{H}^2(\R)$, then 
	for all $\lambda > 0$,
	\begin{eqnarray}\label{gaglia}
	\|A_1(t) \,\phi\|_{\mathtt{L}^2(\R)} & \leq &  \|b\|_{\infty} \|\phi^\prime\|_{\mathtt{L}^2(\R)} \nonumber\\
	&\leq & \|b\|_\infty\,\|\partial_x \left(-\mu_0\,\partial_x^2 +
	\lambda\right)^{-1}\|_{\mathcal{B}(\mathtt{L}^2(\R))}\,\|\left( -\mu_0\,\partial_x^2 + \lambda\right) \phi\|_{\mathtt{L}^2(\R)} \nonumber \\
	& \leq &  \dfrac{\|b\|_\infty}{2 \sqrt{\mu_0}\,\lambda^{1/2}}\,\|\left( -\mu_0\,\partial_x^2 +
	\lambda\right) \phi\|_{\mathtt{L}^2(\R)}  \nonumber\\
	&\leq & \dfrac{\|b\|_\infty}{2 \sqrt{\mu_0}\,\lambda^{1/2}}\,\left( \|-\mu_0\,\phi^{\prime\prime}\|_{\mathtt{L}^2(\R)} +
	\lambda\,\|\phi\|_{\mathtt{L}^2(\R)}\right),
	\end{eqnarray}
	where $\mu_0$ is expressed by (\ref{A1}). Condition (\ref{A1}) implies that
	
	\begin{equation}\label{lapla}
	\|-\mu_0\,\phi^{\prime\prime}\|_{\mathtt{L}^2(\R)} \leq \|A_0(t)\,\phi\|_{\mathtt{L}^2(\R)} \leq \|-\mu_1\,\phi^{\prime\prime}\|_{\mathtt{L}^2(\R)},\;\;\forall \phi \in \mathtt{H}^2(\R).
	\end{equation}
	Substituting (\ref{lapla}) into (\ref{gaglia}), we obtain
	\begin{equation}\label{relati}
	\|A_1(t)\,\phi\|_{\mathtt{L}^2(\R)} \leq  \dfrac{\|b\|_\infty}{2 \sqrt{\mu_0}}\,\left( \dfrac{1}{\lambda^{1/2}}\|A_0(t)\,\phi\|_{\mathtt{L}^2(\R)} +
	\lambda^{1/2}\,\|\phi\|_{\mathtt{L}^2(\R)}\right).
	\end{equation}
	According to \cite[Ch.IV, Theorem 1.1]{Kato2}, estimate  (\ref{relati}) proves that $A_1(t)$ is $A_0(t)$-bound with $A_0(t)$-bound smaller than 1. Further, 
	$A(t) = A_0(t) + A_1(t)$ is closed, and $A(t)$ and $A_0(t)$ have the same domain, $\mathtt{H}^2(\R)$. Thus, $L(t)$ is closable in $\mathtt{L}^2(\R)$, and its closure can be written as 
	\begin{equation}\label{hiperbo}
	A(t) \phi = -a(\cdot,\,t)\,\phi^{\prime\prime} +  b(\cdot,\,t)\,\phi^{\prime}, \;\phi \in \mathtt{H}^2(\R),
	\end{equation}
	where $\phi^\prime,\;\phi^{\prime\prime}$ denotes the distribution sense.
	This completes the proof.
\end{proof}

\begin{lemma}\label{acretivo}
	The operator $A(t)$, defined in (\ref{hiperbo}), is quasi m-accretive in $\mathtt{L}^2(\R)$ and uniformly in $t$. In this case, $A(t) \in G\left(\mathtt{L}^2(\R),\,1,\,\beta\right)$ for some $\beta > 0$.
\end{lemma}

\begin{proof}
	Because we deal with the Hilbert space $\mathtt{L}^2(\R)$, we must show that a 
	constant $\beta > 0$ exists such that
	\begin{itemize}
		\item [(i)] $\langle A(t) \phi, \phi\rangle_{\mathtt{L}^2(\R)} \geq - \beta \| \phi \|^2_{\mathtt{L}^2(\R)},\;\;\forall \phi \in D\left(A(t)\right)=\mathtt{H}^2(\R)$.
		\item [(ii)] $\left( A(t) + \lambda \right)$ is onto $\mathtt{L}^2(\R)$ for some (equivalently for all) $\lambda > \beta$.
	\end{itemize}
	Let $\phi \in \mathtt{H}^2(\R)$; then,
	\begin{align*}
	\langle A(t) &\phi, \phi \rangle_{\mathtt{L}^2(\R)} = \langle{-a(\cdot,\,t) \phi^{\prime\prime}}, \,{\phi}\rangle_{\mathtt{L}^2(\R)} + \langle{b(\cdot,\,t) \phi^\prime}, \,{\phi}\rangle_{\mathtt{L}^2(\R)} \\
	& = -\int_\R a(x,\,t)\,\phi^{\prime\prime}(x) \phi(x) \,dx 
	+ \int_\R b(x,\,t)\,\phi^{\prime}(x) \phi(x) \,dx \\
	& = \int_\R a_x(x,\,t) \phi(x) \phi^\prime(x)\,dx 
	 + \int_\R a(x,\,t) \left(\phi^\prime\right)^2\,dx \\
	& - \dfrac{1}{2} \int_\R b_x(x,\,t) \phi^2(x)\,dx \\
	&	= -\dfrac{1}{2} \int_\R a_{xx}(x,\,t) \phi^2(x)\,dx + \int_\R a(x,\,t) \left(\phi^\prime(x)\right)^2\,dx \\
	&   - \dfrac{1}{2} \int_\R b_x(x,\,t) \phi^2(x)\,dx 
	\geq   -\dfrac{1}{2} \left(\| a_{xx}(\cdot,\,t)\|_\infty + \|b_x(\cdot,\,t)\|_\infty\right)\, \langle{\phi},\,{\phi}\rangle.
	\end{align*}
	Subsequently, from (\ref{A1}) and (\ref{A2}), it follows that
	\begin{equation}\label{acret}
	\beta =\dfrac{1}{2} \left(\| a_{xx}\|_{\mathtt{L}^\infty(\Omega_T)} + \|b_x\|_{\mathtt{L}^\infty(\Omega_T)}\right),
	\end{equation}
	is independent of $t$, proving that $(i)$. 
	
	To prove $(ii)$, the fact that $A(t)$ is a closed operator combined with the inequality in $(i)$
	shows that $\left(A(t) + \lambda \right)$ has a closed range for all $\lambda > \beta$. Thus, it is sufficient to prove that
	$\left(A(t) + \lambda \right)$ has a dense range for $\lambda > \beta$  and sufficient to demonstrate that $R\left( A(t) + \lambda\right)^\bot = \left\{ 0\right\}$.
	
	Let $\phi \in \mathtt{L}^2(\R)$ such that
	$\langle{ \left(A(t) + \lambda \right) \psi}, \,{\phi}\rangle_{\mathtt{L}^2(\R)} = 0$ for all $\psi \in D\left(A(t)\right)$.
	
	Therefore, from the definition of the adjoint operator and property of $A(t)$, we obtain
	\begin{equation}\label{adjun}
	S(t) \psi = - \left(a(\cdot,\,t)\,\psi\right)_{xx} - \left(b(\cdot,\,t)\,\psi\right)_x,
	\end{equation}
	where $D(S(t))=C_0^\infty(\mathbb R)$ is the core of the adjoint operators of $A(t)$ and $D\left(A^*(t)\right)= \mathtt{H}^2(\R)$.
	Let $\phi \in Ker\left( A^*(t) + \lambda\right)$. Then,
	multiplying equation $- \left(a(x,\,t)\,\phi\right)_{xx} - \left(b(x,\,t)\,\phi\right)_x + \lambda\, \phi= 0$ by $\phi$, integrating by parts, and applying the inequality in $(i)$, we obtain
	\begin{eqnarray*}
		0 &=& \langle{\left(a(\cdot,\,t)\,\phi\right)_{xx} - \left(b(\cdot,\,t)\,\phi\right)_x + \lambda \phi}, \,{\phi}\rangle_{\mathtt{L}^2(\R)}=   \\
		&=&  \langle{\phi}, \,{A(t)\phi}\rangle_{\mathtt{L}^2(\R)} + \lambda \,\|\phi\|_{\mathtt{L}^2(\R)}^2 \\
		&\geq &  \left(\lambda - \beta \right)\, \|\phi\|_{\mathtt{L}^2(\R)}^2
	\end{eqnarray*}
	for all $\lambda > \beta$. Because $\left(\lambda - \beta\right) > 0$, we conclude that 
	$Ker( A^*(t) + \lambda) = R ( (A(t) + \lambda)^\bot = \{ 0 \}$, and the proof is finished.
\end{proof}

Let the homogeneous equation be
\begin{equation}\label{homog}
\partial_{ t} u + A(t) u = 0,
\end{equation}
where $A(t)$ is defined as in (\ref{hiperbo}).

We now state the existence of an evolution operator for (\ref{homog}).
\begin{lemma}\label{propagador}
	We assume that Hy1 and Hy2 are satisfied. Then, there exists a unique family of evolution operator $U(t, t')$, defined by
	$$
	(t,\,t^\prime) \in \triangle = \left\{(t,\,t^\prime)\colon 0 \leq t^\prime \leq t \leq T\right\}\longmapsto,
	U(t,\,t^\prime) \in \mathcal{B}\left(\mathtt{L}^2(\R)\right),
	$$
	for any fixed $T > 0$ such that
	\begin{itemize}
		\item [(i)] $(t,\,t^\prime) \longmapsto U(t,\,t^\prime) \in \mathcal{B}\left(\mathtt{L}^2(\R)\right)$ is strongly continuous
		and $U(t,\,t) = I$ for all $t \in [0,\,T]$;
		\item [(ii)] $U(t,\,t^{\prime\prime})= U(t,\,t^\prime)\,U(t^\prime,\,t^{\prime\prime})$ for all $t$, $t^\prime$,
		$t^{\prime\prime}$ such that $ 0 \leq t^{\prime\prime} \leq t^\prime \leq t \leq T$;
		\item [(iii)] $U(t,\,t^\prime)\left(\mathtt{H}^2(\R)\right) \subset \mathtt{H}^2(\R^2)$ and $(t,\,t^\prime) \longmapsto
		U(t,\,t^\prime) \in \mathcal{B}\left(\mathtt{H}^2(\R)\right)$ is strongly continuous in $\mathtt{H}^2(\R)$;
		\item [(iv)] $\displaystyle\partial_t U(t,\,t^\prime)\,\phi = - A(t)\,U(t,\,t^\prime)\,\phi$, \; $\forall \phi
		\in \mathtt{H}^2(\R)$, \; $0 \leq t \leq t^\prime \leq T$;
		\item [(v)]  $\displaystyle\partial_s U(t,\,s)\,\phi = U(t,\,s)\, A(s)\,\phi$, \, $\forall \phi \in \mathtt{H}^2(\R)$, $0\leq s \leq t \leq T$.
	\end{itemize}
\end{lemma}
\begin{proof}

	From Lemma~\ref{evol-operator}, it is sufficient to prove that triplet 
	$\big\{A(t);\,\mathtt{L}^2(\R),$ $\mathtt{H}^2(\R)\big\}$ with the family
	$A(t)$ defined in (\ref{hiperbo}), is a CD-{\it system}.  Lemma~\ref{closure} and Lemma \ref{acretivo} imply the conditions $(i)$ and $(ii)$ of Definition~(\ref{CD}). Condition (iii)
	of this definition remains to be proven.	
	Thus, let the family $\left\{ G(t) \colon t \in [0,\,T]\right\}$ be defined as
	$$
	G(t)\phi = -\widetilde{a}(\cdot,t) \phi^{\prime\prime} + \widetilde{b}(\cdot,t) \phi^\prime,\;\;\phi \in \mathtt{H}^2(\R),
	$$
	where $\widetilde{a}$ and $\widetilde{b}$ are given by Hy2 (ii). We otain
	$G(t) \in \mathcal{B}(\mathtt{H}^2(\R),$ $\mathtt{L}^2(\R))$ for a.e. $t \in [0,\,T]$, is strongly measurable. 
	$\|G(t)\|_{\mathcal{B}\left(\mathtt{H}^2(\R),\,\mathtt{L}^2(\R)\right)}$ is essentially bounded by $t$.
	$$
	A(t) \phi = A(0) \phi + \int_0^t\,G(s) \phi \,ds,\;\mathrm{for}\;\;\phi\in \mathtt{H}^2(\R),
	$$
	This completes the proof.
\end{proof}

Note that, if the source function $f$ does not depend on $u$, then equation \eqref{scalar-local-eq} becames linear as
\begin{equation}\label{eq-linear}
\partial_t u - a(x,\, t)\, u_{xx} + b(x,\,t)\,u_x = f(x,\,t),\,\,x \in \R, \,\, t \in (0, T].
\end{equation}

In this case, if $u$ is a solution to \eqref{eq-linear} with the initial condition,
$u(x, 0) = \phi(x)$, $x \in \R$, then it is easy to see that
\begin{equation}\label{linearnohomo}
u(x, t) = U(t,\,0) \,\phi(x) + \int_0^t\,U(t,\,s)\,f(x, s)\,\mathrm{ds},
\end{equation}
where $U(t,\,s)$ is the evolution operator associated with $A(t)$ defined in (\ref{hiperbo}).
We  call  $u$ given by (\ref{linearnohomo}) a {\it mild solution} of (\ref{eq-linear}).

The proof of the following theorem is analogous to that in \cite[Theorem~2]{Kato4}. 
\begin{theorem}\label{Linear}
	Let Hy1 and Hy2 be satisfied.  If $\phi \in \mathtt{H}^2(\R)$, then
	$f \in C\left([0,\,T],\,\mathtt{L}^2(\R)\right) \cap \mathtt{L}^1\left([0,\,T],\,\mathtt{H}^2(\R)\right)$,
	then  $u$  given by (\ref{linearnohomo}) is a unique solution
	$u\in C\left([0,\,T],\,\mathtt{H}^2(\R)\right) \cap C^1\left([0,\,T],\,\mathtt{L}^2(\R)\right)$ in (\ref{eq-linear}), with $u(x, 0) = \phi(x)$.
	
	Furthermore, we obtain the estimates
	\begin{eqnarray}
	\| u\|_{\infty,\,\mathtt{L}^2(\R)} & \leq & \| U\|_{\infty,\,\mathtt{L}^2(\R)} \,\left( \|\phi\|_{\mathtt{L}^2(\R)} +
	\| f \|_{1,\,\mathtt{L}^2(\R)} \right),\label{estimate1} \nonumber \\
	\| u\|_{\infty,\,\mathtt{H}^2(\R)}  & \leq & \| U\|_{\infty,\,\mathtt{H}^2(\R)} \,\left( \|\phi\|_{\mathtt{H}^2(\R)} +
	\| f \|_{1,\,\mathtt{H}^2(\R)} \right), \label{estimate2}\\
	\|\partial_t u\|_{\infty,\,\mathtt{L}^2(\R)}  & \leq & \| f\|_{\infty,\,\mathtt{L}^2(\R)}  +
	\|A\|_{\infty,\,\mathcal{B}(\mathtt{H}^2(\R),\,\mathtt{L}^2(\R))}\,
	\left( \|\phi\|_{\mathtt{H}^2(\R)} + \right.\nonumber\\
	& & \left. \| f \|_{1,\,\mathtt{H}^2(\R)} \right). \nonumber
	\end{eqnarray}
\end{theorem}

In the proof of the main theorem in this section, two simple lemmas are used, as shown in  \cite{Kato4}.

Let $X$ and $Y$ be reflexive Banach spaces such that $Y \subset X$ is continuously and densely embedded.
\begin{lemma}\label{closed}
	If a subset of $Y$ is convex, closed, or bounded, then it is also closed in $X$.
\end{lemma}

\begin{lemma}\label{continuous}
	If a function $g$ on $[0,\,T]$ to $Y$ is bounded in $Y$-norm and continuous in $X$-norm, then $g$ is weakly continuous (hence, strongly measurable) as a $Y$-valued function.
\end{lemma}

The main result of this section is provided in the next theorem, which follows similar ideas:
as in the proof in \cite[Theorem 6]{Kato4}. For an existence solution in Sobolev spaces using similar arguments, see, for example, \cite{Cunha}.
\begin{theorem} [Local Solution]\label{scalar-local-theo}
	Let Hy1, Hy2, and Hy3 be satisfied.  If $\phi \in W$, then equation
	\eqref{scalar-local-eq} has a unique solution
	\begin{equation}\label{condi}
	u \in C([0,\,T^\prime],\,W') \cap C^1 ( [0,\,T^\prime],\,\mathtt{L}^2(\R)),\;\;\mathrm{with}\;\;u(0) = \phi,
	\end{equation}
	for some $T^\prime$, $ 0 < T^\prime \leq T$, and the same open ball $W'\supset W$,
	which is given explicitly by 
	\begin{align}\label{mild-solution}
	u(t)=U(t,0)\phi+\int_0^t  U(t,\tau)f(u(\tau))d\tau, \,\,\,t \in [0, T'].
	\end{align}
\end{theorem}

\begin{proof}
The first part of the proof deals with existence. 
	
We consider an open ball $W$ centered at the origin, with radius $\rho > 0$.
	For constants $T > 0$ and $R>0$ to be defined and $M > \rho$, we define the space
	\begin{align*}
	E_T=\{u\in C([0, T],\mathtt{L}^2(\R)): \,\, 
	&\sup_{[0,T]}\|u(t)\|_{\mathtt{H}^2(\R)}\leq R, \, \mbox{and} \\ 
	&\sup_{[0,T]}\|u(t)-U(t,0)\phi\|_{\mathtt{L}^2(\R)} \leq M\}.
	\end{align*}
		
	First, from Lemma \ref{acretivo}, we choose 
	
	\begin{equation}\label{M}
	T\leq \frac{1}{\beta}\ln (\rho^{-1}M), 
	\end{equation}
	where $\beta$ is given by \eqref{acret}.
	The application $g(t)\equiv 0$ satisfies $g\in E_T$. Thus, $E_T$ is not empty.

	From Lemma \ref{closed}, it follows that $E_T$ is a complete metric space with the 	metric
	\begin{equation*}
	d(u,v)=\sup_{[0,T]}\|u(t)-v(t)\|_{\mathtt{L}^2(\R)}. 
	\end{equation*}
	
	For all $u\in E_T$, where $f(u(\tau)):=f(\cdot,\tau,u(\tau))$, the weak continuity of $f(u(\cdot))$ in the ${\mathtt{H}^2(\R)}$-norm follows Hy3 (i), Hy3 (ii), and Lemma~\ref{continuous}.
	
	Defining
	\begin{equation}\label{function-contraction}
	\Phi u (t)=U(t,0)\phi+\int_0^t  U(t,\tau)f(u(\tau))d\tau, \quad t\in [0,T],
	\end{equation}
	it is shown that there is $0<T'\leq T$ such that $\Phi$ is a contraction in $E_{T'}$. 
	
If $u\in E_T$ and $t_1<t$, then
	\begin{equation}\label{sigma}
	\begin{split}
	\Phi u(t)-\Phi u(t_1)=& \ U(t,0)\phi-U(t_1,0)\phi+\int_0^{t_1}(U(t,\tau)-U(t_1,\tau))f(u(\tau))d\tau\\
	&+\int_{t_1}^{t}U(t,\tau)f(u(\tau))d\tau\\
	:=& \ \sigma(t,t_1)+\bar{\sigma}(t,t_1)+\tilde \sigma(t,t_1).
	\end{split}
	\end{equation}
	
	Lemma \ref{propagador}(i) implies that $\sigma(t,t_1)\to 0$, when $t_1\to t$.
	
	From $\mathtt{H}^2(\R)\hookrightarrow \mathtt{L}^2(\R)$ and Hy3 (i), it follows that 
	\begin{equation}\label{imersao}
	\|f(u(\tau))\|_{\mathtt{L}^2(\R)}\leq \|f(u(\tau))\|_{\mathtt{H}^2(\R)}\leq \mu, 
	\end{equation}
	where $\mu$ depends on $R$. Hence, Lemmas \ref{acretivo} and \eqref{imersao} imply that 
	\begin{equation*}
	\begin{split}
	\|\tilde \sigma(t,t_1)\|_{\mathtt{L}^2(\R)} 
	&\leq \int_{t_1}^t \|U(t,\tau)f(u(\tau))\|_{\mathtt{L}^2(\R)} d\tau\\
	&\leq e^{\beta T}\int_{t_1}^t \|f(u(\tau))\|_{\mathtt{L}^2(\R)} d\tau\\
	&\leq \mu (t-t_1) e^{\beta T}\to 0, \,\,\,\mbox{when} \,\,\,t_1\to t.  
	\end{split}
	\end{equation*}
		
	From Hy3 (ii), it follows that $f(u(\tau))\in \mathtt{L}^2(\R)$ for all $\tau \in [0,T]$. Thus, Lemma \ref{propagador} (i) implies that 	
	\begin{equation*}
	\|(U(t,\tau)-U(t_1,\tau))f(u(\tau))\|_{\mathtt{L}^2(\R)}\to 0, \quad \mbox{when} \quad t_1\to t,
	\end{equation*}
	for a.e. $\tau \in [0,T]$.
	
	In addition, from Lemmas \ref{acretivo} and  \eqref{imersao},
	\begin{equation*}
	\|(U(t,\tau)-U(t_1,\tau))f(u(\tau))\|_{\mathtt{L}^2(\R)}\leq 2e^{\beta T} \|f(u(\tau))\|_{\mathtt{H}^2(\R)}\leq 2\mu e^{\beta T}, \quad \tau \in [0,T].
	\end{equation*}
	Then, the dominant convergence theorem implies that
	$$\bar{\sigma}(t,t_1)\to 0,\,\,\,\mbox{when} \,\,\,t_1\to t.$$
	
	Thus, $\Phi u$ is left continuous in $t$. Right continuity  of 
	$\Phi u$ can be obtained similarly. Therefore, $\Phi u \in C([0, T],\mathtt{L}^2(\R))$,
	
	Now, by taking the $\mathtt{H}^2(\R)$-norm in \eqref{function-contraction} and using Hy3 (i) and Lemma~\ref{propagador}(iii), we obtain 
	\begin{equation}\label{contE2}
	\begin{split}
	\|\Phi u (t)\|_{\mathtt{H}^2(\R)} &\leq \|U(t,0)\phi\|_{\mathtt{H}^2(\R)}+
	\int_0^t \|U(t,\tau)f(u(\tau))\|_{\mathtt{H}^2(\R)} d\tau\\
	&\leq e^{\tilde \beta T}\|\phi\|_{\mathtt{H}^2(\R)} + 
	e^{\tilde \beta T}\int_0^t \|f(u(\tau))\|_{\mathtt{H}^2(\R)} d\tau\\
	& \leq e^{\tilde \beta T}(\rho+T'\mu).
	\end{split}
	\end{equation}
	
	Thus, Lemmas \ref{acretivo} and \eqref{imersao} imply that
	\begin{equation}\label{contE1}
	\begin{split}
	\|\Phi u (t)-U(t,0)\phi\|_{\mathtt{L}^2(\R)} &\leq \int_0^t \|U(t,\tau)f(u(\tau))\|_{\mathtt{L}^2(\R)} d\tau\\
	&\leq e^{\beta T} \int_0^t \|f(u(\tau))\|_{\mathtt{L}^2(\R)} d\tau\\
	& \leq T' \mu e^{\beta T},
	\end{split}
	\end{equation}
	and Lemma \ref{acretivo} and Hy3 (iii) imply that
	\begin{equation}\label{cont}
	\begin{split}
	\|\Phi u (t)-\Phi v (t)\|_{\mathtt{L}^2(\R)} &\leq \int_0^t \|U(t,\tau)(f(u(\tau))-
	f(v(\tau)))\|_{\mathtt{L}^2(\R)} d\tau\\
	&\leq e^{\beta T} \int_0^t \|f(\cdot,\tau,u(\tau))-f(\cdot,\tau,v(\tau))\|_{\mathtt{L}^2(\R)} d\tau\\
	& \leq  \kappa e^{\beta T}\int_0^t \|u(\tau)-v(\tau)\|_{\mathtt{L}^2(\R)} d\tau\\
	& \leq T' \kappa e^{\beta T} d(u,v).
	\end{split}
	\end{equation}
	We define the open ball $W' \supset W$ with radius $R$ such that
	\begin{equation}\label{cond R M}
	R > \rho \big(M \rho^{-1}\big)^{\frac{\tilde \beta}{\beta}},
	\end{equation}
	and by using \eqref{M} and \eqref{cond R M}, we see that $R>\rho e^{\tilde \beta T}$.
	Thus, by combining \eqref{contE2}, \eqref{contE1}, and \eqref{cont}, and taking 
	$T^\prime$ that satisfies
	\begin{equation}
	0<T'<\min\Big \{T,\frac{M}{ \mu e^{\beta T}}, \frac{1}{\kappa e^{\beta T}}, \frac{1}{\mu}\Big(\frac{R}{e^{\tilde \beta T}}-\rho\Big)\Big\},
	\end{equation}
	implies that $\Phi : E_{T'}\to E_{T'}$ is a contraction. Therefore, from the Banach fixed-point theorem, equation (\ref{scalar-local-eq}) has a unique mild solution $u$ with $u(0) = \phi$, which is given by \eqref{mild-solution}.
	A straightforward calculation shows that this mild solution satisfies $u\in C([0,T'],W')$ and $\partial_t u \in C([0,T'],\mathtt{L}^2(\R))$, indicating the existence of a solution.  

	Rest to prove the uniqueness.
	
	Let $u$ and $v$ be solutions of \eqref{scalar-local-eq}, satisfying \eqref{condi}  such that $u(0)=v(0)=\phi$. For all $t\in [0,T']$, using \eqref{mild-solution}, Lemma~\ref{acretivo}, and Hy3 (iii), it follows that
	\begin{equation*}
	\begin{split}
	\|u(t)-v(t)\|_{\mathtt{L}^2(\R)} &\leq \int_0^t \|U(t,\tau)(f(u(\tau))-f(v(\tau)))\|_{\mathtt{L}^2(\R)} d\tau\\
	&\leq e^{\beta T'}\int_0^t \|f(\cdot,\tau,u(\tau))-f(\cdot,\tau,v(\tau))\|_{\mathtt{L}^2(\R)} d\tau\\
	&\leq \kappa e^{\beta T'}\int_0^t \|u(\tau)-v(\tau)\|_{\mathtt{L}^2(\R)} d\tau.
	\end{split}
	\end{equation*} 
Therefore, Gronwall's lemma implies that $u(t)=v(t)$, completing the proof.
\end{proof}

\section{Global solution for a scalar equation}\label{scalar-global}

Here, we consider the second-order semilinear scalar equation  
\begin{equation}\label{scalar-global-eq}
\partial_t u - a(x,\, t)\, u_{xx} + b(x,\,t)\,u_x = f(x,\,t,\,u),\,\,x \in \R, \,\, t > 0,
\end{equation}
where $a,\,\, b\, :\Omega = \R \times [0, \infty) \to \R$ and $f : \Omega \times \R \to \R$ are the given functions. We also consider the
family of associated differential operators, defined by 
\begin{equation}\label{ope-LG}
\left(L_G(t)\phi\right)(x) = -a(x,t) \phi^{\prime\prime}(x) + b(x,t) \phi^\prime(x),\;\;\phi \in C^{\infty}_0(\R),\,\,t \in [0, \infty).
\end{equation}

To solve \eqref{scalar-global-eq}, it is necessary to establish hypotheses similar to 
for the local solution in \ref{scalar-local}.

\begin{itemize} 
	\item [] {\bf (Hy4):} The coefficients of the operator $L_G(t)$ satisfy $Hy1$ and $Hy2$,
given in \ref{scalar-local}, when the interval $[0, T]$ is substituted by $[0, \infty)$.
	\item [] {\bf (Hy5):} For $w$ belonging to some open ball $W \subset \mathtt{H}^2(\R)$, the function 
	$(t,\,w) \mapsto f(t, \,w)$ satisfies the same items (i), (ii), and (iii),
    as in Hy3, when interval $[0, T]$ is substituted by $[0, \infty)$.	
\end{itemize}

The proof of the next two lemmas is similar to that of Lemmas~\ref{closure} and \ref{acretivo}, respectively, where the interval $[0, T]$ is substituted by $[0, \infty)$.

\begin{lemma}\label{closureG}
	Assume that the operator $L_G(t)$ satisfies $Hy4$. Subsequently, $L_G(t)$  is
	closable in $\mathtt{L}^2(\R)$ and the domain $D\left(A_G (t)\right)= \mathtt{H}^2(\R)$ of $A_G(t)$, the closure of $L_G(t)$ is independent of $t$.
\end{lemma}

\begin{lemma}\label{acretivoG}
	The operator $A_G (t)$, defined in Lemma~\ref{closureG} is quasi m-accretive in $\mathtt{L}^2(\R)$ and uniform in $t$; that is, there is a constant $\beta_G > 0$ such that
	\begin{itemize}
		\item [(i)] $\langle{A_G(t) \phi}, \,{ \phi}\rangle_{\mathtt{L}^2(\R)} \geq - 
		\beta_G \| \phi \|^2_{\mathtt{L}^2(\R)},\;\;\forall \phi \in D\left(A_G(t)\right)=\mathtt{H}^2(\R)$,
		\item [(ii)] $\left( A_G(t) + \lambda \right)$ is onto $\mathtt{L}^2(\R)$ for some (equivalently for all) $\lambda > \beta_G$.
	\end{itemize}
	
	Here,
	$A_G(t) \in G\left(\mathtt{L}^2(\R),\,1,\,\beta_G\right)$, where
	\begin{equation}\label{acret1}
	\beta_G =\dfrac{1}{2} \left(\| a_{xx}\|_{\mathtt{L}^\infty(\Omega)} + \|b_x\|_{\mathtt{L}^\infty(\Omega)}\right).
	\end{equation}
\end{lemma}
 
In addition to Lemmas \ref{closureG} and \ref{acretivoG}, the next regularity lemma is fundamental to the proof of global existence. 
\begin{lemma}[Regularity] \label{regularidade}
	Let $A_G(t)$ be the operator defined in Lemma~\ref{closureG}, then
	\begin{equation}\label{regu0}
	\langle{A_G(t)\,\phi^\prime},\,{\phi^\prime}\rangle_{\mathtt{L}^2(\R)} \geq -C\,\dfrac{\beta_G^2}{\mu_0}\,\|{\phi}\|^2_{\mathtt{L}^2(\R)},\;\;\forall \phi \in \mathcal{S}(\R).
	\end{equation}
\end{lemma}

\begin{proof}
	Let $\phi \in \mathcal{S}(\R)$, then
	\begin{align}\label{GN}
	\langle{A_G(t) \phi^\prime },\,{\phi^\prime}\rangle_{\mathtt{L}^2(\R)} &= \langle{-a(\cdot,\,t) \phi^{\prime\prime\prime}},\,{\phi^\prime}\rangle_{\mathtt{L}^2(\R)} + \langle{b(\cdot,\,t) \phi^{\prime\prime}},\,{\phi^\prime}\rangle_{\mathtt{L}^2(\R)} 
	\nonumber\\
	 = &-\int_\R a(x,\,t)\,\phi^{\prime\prime\prime}(x) \phi^\prime(x) \,dx + \int_\R b(x,\,t)\,\phi^{\prime\prime}(x) \phi^\prime(x) \,dx   
	\nonumber\\
	 = & \int_\R a_x(x,\,t) \phi^\prime(x) \phi^{\prime\prime}(x)\,dx + \int_\R a(x,\,t) \left(\phi^{\prime\prime}(x)\right)^2\,dx 
	\nonumber \\
	-&\dfrac{1}{2} \int_\R b_x(x,\,t) \left(\phi^\prime(x)\right)^2(x)\,dx  
	 = -\dfrac{1}{2} \int_\R a_{xx}(x,\,t)\left( \phi^\prime\right)^2(x)\,dx 
	\nonumber \\ 
	 +& \int_\R a(x,\,t) \left(\phi^{\prime\prime}(x)\right)^2\,dx  
	 -\dfrac{1}{2} \int_\R b_x(x,\,t) \left(\phi^\prime(x)\right)^2\,dx 
	\nonumber \\
	 \geq & \mu_0\,\int_\R \left(\phi^{\prime\prime}(x)\right)^2\,dx - \dfrac{1}{2} \int_\R \left(a_{xx}(x,\,t) + b_x(x,\,t)\right)\, \left(\phi^\prime(x)\right)^2\,dx 
	\nonumber \\
	 \geq & \mu_0\,\|{\phi^{''}}\|^2_{\mathtt{L}^2(\R)}   -\beta_G\,\|{\phi^\prime}\|^2_{\mathtt{L}^2(\R)}\,.
	\end{align}
	Using the Gagliardo-Nirenbreg inequality, we obtain
	\begin{equation}\label{Ineq}
	\|{\phi^\prime}\|_{\mathtt{L}^2(\R)} \leq c\,\|{\phi^{\prime\prime}}\|_{\mathtt{L}^2(\R)}^{1/2}\,\|{\phi}\|_{\mathtt{L}^2(\R)}^{1/2},
	\end{equation}
	and by using (\ref{Ineq}) in (\ref{GN}), we have
	\begin{equation}\label{REG}
	\langle{A_G(t) \phi^\prime },\,{\phi^\prime}\rangle_{\mathtt{L}^2(\R)}   \geq  \mu_0\,\|{\phi^{\prime\prime}}\|^2_{\mathtt{L}^2(\R)}   -c\,\beta_G\,\|{\phi^{\prime\prime}}\|_{\mathtt{L}^2(\R)}\,\|{\phi}\|_{\mathtt{L}^2(\R)}.
	\end{equation}
	
	For a positive constant  $\alpha$, we know that
	\begin{equation}\label{REG1}
	\alpha\,\xi - \mu_0\,\xi^2 \leq \dfrac{\alpha^2}{2\mu_0},\;\;\forall\, \xi \geq 0.
	\end{equation}
	
	Thus, combining (\ref{REG}) and (\ref{REG1}), we obtain inequality \eqref{regu0}.
\end{proof}

The global solution theorem is now stated.

\begin{theorem} [Global solution]\label{scalar-global-theo}
	
	We assume that Hy4 and Hy5 are satisfied.  If $\phi \in W$, then equation
	\eqref{scalar-global-eq} has a unique solution
	\begin{equation}\label{condi1}
	u \in C([0,\,\infty), \,W') \cap C^1 ( [0,\,\infty), \,\mathtt{L}^2(\R)),\;\;\mathrm{with}\;\;u(0) = \phi,
	\end{equation}
	for the same open ball, $W'\supset W$.
\end{theorem}

\begin{proof}
Assuming that $u$ is the local solution with $u(0)=\phi$, given by Theorem \ref{scalar-local-theo} and, we define 
\begin{align} \label{uniqueT}
T^* = 
\sup \big\{ T > 0, \,\,\, \mbox{such that the local solution $u$
  is defined in $[0,\,T]$} \big\}\,.
\end{align}
Meanwhile, we assume $T^* < \infty$. Let $A_{G}(t)$ be the operator defined in lemma ~\ref{closureG}. Then, 
 equation ~\eqref{scalar-global-eq} is equivalent to 	
	\begin{equation}\label{scalar-AG-eq}
	\partial_t u+A_{G}(t)u=f(t,u).
	\end{equation}
	From Hy5, we obtain
	\begin{equation}\label{f1g}
	\|f(t, \,u)\|_{\mathtt{H}^2(\R)}\leq \mu,\,\,\mbox{for} \,\,\,t \in [0, \infty),\,\,\mbox{and}\,\,\,u \in W,
	\end{equation}
	where $W$ is an open ball in $\mathtt{H}^2(\R)$ and $\mu$ is a constant independent of $t$.
	
	By applying the operator $\partial_x$ on both sides of \eqref{scalar-AG-eq}, we obtain
	\begin{equation}\label{vx}
	\partial_t v +  A_{G_1}(t)\,v + b_x(\cdot,\,t)\,v = \partial_x f(\cdot, t, u),
	\end{equation}
	where $v=\partial_x u$, $A_{G_1}(t)$  is defined as in Lemma~\ref{closureG}, where $a_1(x,\,t)= a(x,\,t)$ and $b_1(x,\,t)=b(x,\,t) -  a_x (x,\,t)$. Setting $z=\partial_x^2 u$, from \eqref{vx}, we obtain
	\begin{equation}\label{vxx}
	\partial_t z +  A_{G_2}(t)\,z + \left(2\,b_x(\cdot,\,t) - a_{xx}(\cdot,\,t)\right)\,z + 
	b_{xx}(\cdot,\,t)\,\partial_x v = \partial_x^2 f(\cdot, t, u),
	\end{equation}
	where  $A_{G_2}(t)$ is defined as in Lemma~\ref{closureG}, $a_2(x,\,t)= a(x,\,t)$, and $b_2(x,\,t)=b(x,\,t) - 2 a_x (x,\,t)$.
	
	Using equations \eqref{scalar-AG-eq}, \eqref{f1g}, and Lemma~\ref{acretivoG}, we obtain 	
\begin{equation}
\begin{split}\label{v}
\frac12 \frac{d}{dt}\|{u(t)}\|^2_{\mathtt{L}^2(\R)}&=-\langle{A_{G}(t)u},u\rangle_{\mathtt{L}^2(\R)}+\langle{f(t,u)}, {u}\rangle_{\mathtt{L}^2(\R)}\\
&\leq (\beta_G + \mu)\|{u}\|^2_{\mathtt{L}^2(\R)}. 
\end{split}
\end{equation}
	
	From \eqref{vxx}, Lemma~\ref{regularidade}, and \eqref{f1g}, we obtain
	
	\begin{align} \label{w}
	\frac12 &\frac{d}{dt}\|{z(t)}\|^2_{\mathtt{L}^2(\R)}=-\langle{A_{G_2}(t)z}, z\rangle_{\mathtt{L}^2(\R)}-\langle{(2b_x-a_{xx})z-b_{xx}\partial_x u},\,z\rangle_{\mathtt{L}^2(\R)} 
	\nonumber \\
	&+ \langle{\partial_x^2 f}, z\rangle_{\mathtt{L}^2(\R)} 
	 \leq const\, \beta_{G_2}^2\|\partial_x u\|_{\mathtt{L}^2(\R)}^2+\big(2\|b_x\|_{\mathtt{L}^\infty(\Omega)}+
	\|a_{xx}\|_{\mathtt{L}^\infty(\Omega)}\big)\|z\|_{\mathtt{L}^2(\R)}^2
	\nonumber \\ 
	& + \|b_{xx}\|_{\mathtt{L}^\infty(\Omega)}\|\partial_x u\|_{\mathtt{L}^2(\R)}\|z\|_{\mathtt{L}^2(\R)} + \mu\|{z}\|_{\mathtt{L}^2(\R)}.
	\end{align}
	The constants $\beta_{G}$ and $\beta_{G_2}$ are defined in Lemma \ref{acretivoG}.
	Thus, as 
	\begin{equation}
	\|{\partial_x u}\|_{\mathtt{L}^2(\R)} \leq const\, \|{u}\|_{\mathtt{L}^2(\R)}+\|{\partial_x^2 u}\|_{\mathtt{L}^2(\R)},
	\end{equation}
	inequalities \eqref{v}, \eqref{w}, and Gronwall's lemma imply that
	\begin{align}\label{desv}
	&\|{u(t)}\|_{\mathtt{L}^2(\R)} \leq  \, \|{u(0)}\|_{\mathtt{L}^2(\R)}\,e^{(\beta_G + \mu)t}, \\ 
	& \mbox{and} \nonumber \\
	&\|{u_{xx}(t)}\|_{\mathtt{L}^2(\R)} \leq \,\|{u_{xx}(0)}\|_{\mathtt{L}^2(\R)}\,C\big(\mu,\|{\phi}\|_{\mathtt{L}^2(\R)},\, \|{\psi}\|_{\mathtt{L}^2(\R)},\beta_{G},\,\beta_{G_2},\,\|{b_x}\|_{\mathtt{L}^2(\Omega)},\,\nonumber \\ 
	& \phantom{-----}\|{b_{xx}}\|_{\mathtt{L}^2(\Omega)},\,t\big),\label{desv2}
	\end{align}
	for all $t\in [0,\,T^*)$, where $C$ is a continuous function of $t$.
	Therefore,  inequalities \eqref{desv} and \eqref{desv2} imply that
	\begin{equation}\label{finalestima}
	\|{u(t)}\|_{\mathtt{H}^2(\R)} \leq \|{u(0)}\|_{\mathtt{H}^2(\R)} \,\Psi(t),\;\;\forall t \in [0,\,T^*),
	\end{equation}
	where $\Psi$ is a uniformly continuous function of $t$. Thus, the local solution $u$ can be extended beyond $T^*$, in contrast to the definition of $T^*$. This completes the proof.
\end{proof}

\end{appendix}

\section*{Homage} 
M. R. Batista, A. Cunha, J. C. Da Mota, and R. A. Santos leave sincere homage to author E. A. Alarcon, who greatly contributed to this article, but unfortunately left us early.

\bibliographystyle{plain}
\bibliography{semilinear}

\end{document}